\documentclass[aos,preprint]{imsart} 

\RequirePackage[OT1]{fontenc}
\RequirePackage{amsthm,amssymb,amsfonts,amsmath}
\RequirePackage[numbers]{natbib}
\RequirePackage[colorlinks,citecolor=blue,urlcolor=blue, pdffitwindow=false, pdfstartview={FitH}]{hyperref}
\usepackage{epsfig}
\usepackage{stmaryrd}
\usepackage{graphics}
\usepackage{subfigure}
\usepackage[T1]{fontenc}
\usepackage{graphicx}
\usepackage{epstopdf}
\usepackage{hyperref}
\usepackage{pdfsync}
\usepackage{color}
\usepackage{dsfont}
\newcommand{\N}{\ensuremath{\mathbb N} }
\newcommand{\R}{\ensuremath{\mathbb R} }
\newcommand{\C}{\ensuremath{\mathbb C} }
\newcommand{\Z}{\ensuremath{\mathbb Z} }

\newcommand{\PR}{\ensuremath{{\mathbb P}} }
\newcommand{\PP}{\ensuremath{{\mathbb P}} }
\newcommand{\EE}{\ensuremath{{\mathbb E}} }

\newcommand{\cA}{{\cal A}}
\newcommand{\cB}{{\cal B}}
\newcommand{\cC}{{\cal C}}

\newcommand{\cE}{{\cal E}}
\newcommand{\cF}{{\cal F}}
\newcommand{\cG}{{\cal G}}
\newcommand{\cH}{{\cal H}}

\newcommand{\cM}{{\cal M}}
\newcommand{\cN}{{\cal N}}
\newcommand{\cO}{{\cal O}}
\newcommand{\cP}{{\cal P}}

\newcommand{\cV}{{\cal V}}

\newtheorem{theo}{Theorem}
\newtheorem{proposition}{Proposition}
\newtheorem{lemma}{Lemma}

\newtheorem{remark}{Remark}
\newtheorem{coro}{Corollary}

\numberwithin{equation}{section}
\numberwithin{ass}{section}
\numberwithin{theo}{section} \numberwithin{proposition}{section}
\numberwithin{lemma}{section}
\numberwithin{remark}{section}

\newcommand{\re}{\Re\mathrm{e}} 
\def\M{\mathfrak{M}([0,1])}

\def \i{\mathfrak{i}}
\def \xikn{\xi_{n}} 

\newcommand{\1}{\ensuremath{\textbf{1}}}

\newcommand{\underbraceabs}[2]{\left|\vphantom{#1}\right. \underbrace{#1}_{#2} \left.\vphantom{#1}\right| }

\arxiv{arXiv:0000.0000}

\startlocaldefs
\numberwithin{equation}{section}
\theoremstyle{plain}

\endlocaldefs

\begin{document}

\begin{frontmatter}
\title{Bayesian methods in the Shape Invariant Model (I): Posterior contraction rates on probability measures}
\runtitle{Bayesian methods in the Shape Invariant Model (I)}

\begin{aug}
\author{\fnms{Dominique} \snm{Bontemps}\thanksref{t2}\ead[label=e1]{dominique.bontemps@math.univ-toulouse.fr}} \and
\author{\fnms{S\'ebastien} \snm{Gadat}\thanksref{t2}\ead[label=e2]{sebastien.gadat@math.univ-toulouse.fr}}

\thankstext{t2}{The authors acknowledge the support of the French Agence Nationale de la Recherche (ANR) under references ANR-JCJC-SIMI1 DEMOS and ANR Bandhits.} 
\runauthor{D. Bontemps and S. Gadat}

\affiliation{Institut Math\'ematiques de Toulouse, Universit\'e Paul Sabatier}

\address{Institut Math\'ematiques de Toulouse, 
 Universit\'e Paul Sabatier\\118 route de Narbonne
F-31062 Toulouse Cedex 9 FRANCE\\
\printead{e1}\\\printead{e2}\\}

\end{aug}

\begin{abstract}

In this paper, we consider the so-called Shape Invariant Model which stands for the estimation of a function $f^0$ submitted to a random translation of law $g^0$ in a white noise model. We are interested in such a model when the law of the deformations is {\em unknown}. 
We aim to recover the law of the process $\PP_{f^0,g^0}$. 

In this perspective, we adopt a Bayesian point of view and find prior on $f$ and $g$ such that the posterior distribution  concentrates at a polynomial rate around $\PP_{f^0,g^0}$   when $n$ goes to $+\infty$. 
We intensively use some Bayesian non parametric tools coupled with mixture models and  believe that some of our results obtained on this mixture framework may be also of interest for  frequentist point of view.
\end{abstract}



\end{frontmatter}

\section{Introduction}

We are interested in this work in the so-called Shape Invariant Model (SIM). Such model aims to describe a statistical process which involves a deformation of a functional shape according to some randomized geometric variability. 
Such geometric deformation of a common unknown shape may be well-suited in various and numerous fields, like image processing (see for instance \cite{amitpiccgre} or \cite{Park}). It  corresponds to a particular case of  the general Grenander's theory of shapes (see \cite{gremil} for a detailed introduction on this topic). 
This kind of model is also useful in medicine:  the recent work of \cite{Bigotecg} deals with the differentiation between normal and arrhythmic cycles in electrocardiogram. 
It appears in genetics if one deals with some delayed activation curves of genes when drugs are administrated to patients, or in Chip-Seq estimation when translations in protein fixation yield randomly shifted counting processes (see for instance \cite{Mortazaviscience} and \cite{BGKM12}). It also occurs in econometric for the analysis of Engel curves \cite{Blundell}, in landmark registration \cite{bigland}\ldots

Such a model has received a large interest in the statistical community as pointed by the large amount of references on this subject. Some works  consider a semi-parametric approach for the estimation ({\it self-modeling regression framework} used by \cite{kg} and \cite{BGV}). 
In \cite{Cast12}, the author applies some Bayesian techniques to obtain also statistical results on SIM in a semi-parametric setting when the level of noise on observations asymptotically vanishes. 
Older approaches use parametric settings (see \cite{glasbey} and the discussion therein for an overview) and study the so-called Fr\'echet mean of pattern. 
Standard $M$-estimation or Bayesian methods are exploited in \cite{BGL09} or \cite{AAT} and same authors develop in \cite{AKT} a nice stochastic algorithm to run estimation in such a model. 
Some recent works follow some testing strategies to obtain curve registration \cite{CD11}, \cite{C12}. At last, note that \cite{BG10} obtains some minimax adaptive results for non-parametric estimations in the Shape Invariant Model when one knows the law of the randomized translations. 

All these works are interested in the statistical process of deformation of the "mean common shape" and generally aim to recover this unknown functional object according to noisy i.i.d. observations. Moreover, the Shape Invariant Model is considered as a standard benchmark for statistical methods which aim to compute estimations in some more general deformable models. 
Of course, the SIM could be extended to some more general situations of geometrical deformations described through an action of a finite dimensional Lie Group (see \cite{BCG} for a precise non parametric description). We have decided to restrict our work here to the simplest case of the one dimensional Lie group of translation $\mathbb{S}^1$ to warp the functional objects.

This work has been inspired by several discussions with Alain Trouv\'e about the work \cite{AKT} for the study of the Shape Invariant Model. We aim to extend their parametric Bayesian framework to the non-parametric setting and then study the behaviour of some posterior distributions. 
Hence,  the motivation of the paper is mainly theoretical: we want to describe the asymptotic evolution of the posterior probability distributions when data are coming from the SIM. Of course, we need to build  suitable prior which yield nice contraction rate for this posterior distribution. 
We have decided to consider the general case where both the functional shape and the probability distribution of the  deformations are unknown. Indeed, it corresponds to the more realistic case. From the best of our knowledge, no sharp statistical results have been derived yet in this non-parametric situation.

Our work will describe the evolution of the posterior distribution when the number of observations grows to $+ \infty$ with a fixed noise level $\sigma$. It is an important difference with the study of the asymptotically vanishing noise situation ($\sigma \rightarrow 0$). It is itself a special feature of the Shape Invariant Model: there is no obvious Le Cam equivalence of experiments (see \cite{LY}) for the SIM between the experiments when $n \mapsto + \infty$ and when $\sigma \mapsto 0$.  It is illustrated by the very different minimax results obtained in \cite{BG10} ($n \mapsto + \infty$) and in \cite{BG12} ($\sigma \mapsto 0$).
We will use in the sequel quite standard Bayesian non parametric methods to obtain the frequentist consistency and some contraction rates of the Bayesian procedures. Such tools rely on some important contributions of \cite{BSW} and \cite{GGvdW00} for the posterior behaviour in general situations, as well as Bayesian properties on mixture models stated in \cite{GvdW01} and \cite{GW}.

The paper is organised as follows. Section \ref{sec:model_not_result} presents a sharp description of the Shape Invariant Model (shortened as SIM in the sequel), as well as standard elements on Bayesian and Fourier analysis. It also provides some notations for mixture models. It ends with the statement of the posterior contraction around the true law on functional curves, which is our main result. Section \ref{sec:metric} provides a metric description of the important probability spaces of the model.
At last, Section \ref{sec:proof_main} presents the proof of this main result. 
We end the paper with numerous challenging issues. 

We gather in the appendix sections some technical points: the metric description of the Shape Invariant Model embedded in a special randomized curves space  and the calibration of suitable priors for the SIM.

\section{Model, notations and main results}  \label{sec:model_not_result}

\subsection{Statistical settings}

\paragraph{Shape Invariant Model}
We recall here the random Shape Invariant Model. We assume $f^0$ to be  a  function  which belongs to a subset $\cF$ of smooth functions. 
We also consider a probability measure $g^0$ which is an element of the set $\mathfrak{M}([0,1])$. This last set stands for the set of probability measures on $[0,1]$. We observe $n$ realizations of  noisy and randomly shifted complex valued curves $Y_1, \ldots, Y_n$ coming from the following white noise model
\begin{equation}\label{eq:model}
\forall x \in [0,1] \quad \forall j=1 \ldots n \qquad dY_j(x): = f^0(x-\tau_j) dx + \sigma dW_j(x).
\end{equation}
Here, $f^0$ is the \textit{mean} pattern of the curves $Y_1, \ldots, Y_n$ although the random shifts $(\tau_j)_{j = 1 \ldots n}$ are sampled independently according to the probability measure $g^0$. 
Moreover, $(W_j)_{j = 1 \ldots n}$ are independent complex standard Brownian motions on $[0,1]$ and model the presence of noise in the observations, the noise level is kept fixed in our study and is set to $1$ for sake of simplicity.

In the sequel, $f^{-\tau}$ will denote the pattern $f$ shifted by $\tau$, that is to say the function $x\mapsto f(x-\tau)$. 
Complex valued curves are considered here for the simplicity of notations. However all our results can be adapted to the simpler case where all curves $Y_j$'s are real valued. 
A complex standard Brownian motion $W_t$ on $[0,1]$ is such that $W_1$ is a standard complex Gaussian random variable, whose distribution is denoted by $\cN_{\C}(0,1)$; a standard complex Gaussian random variable have independent real and imaginary parts with a real centered Gaussian distribution of variance $1/2$.

This work will address the question of the behaviour of some posterior distributions on $\cF\otimes\mathfrak{M}([0,1])$ given some functional  $n$-sample $(Y_1, \ldots, Y_n)$.
Since our work will be mainly asymptotic with $n \rightarrow + \infty$, we  intensively use some standard notation  such as "$\lesssim$" which refers to an inequality up to a multiplicative absolute constant. In the meantime, $a\sim b$  stands for $a/b \longrightarrow 1$.
\paragraph{Bayesian framework}

Since most of statistical works on the SIM are frequentists, we have decided to briefly recall here the Bayesian formalism following the presentation of \cite{GGvdW00}. Familiar readers can thus omit this paragraph.

 Functional objects $f^0$ and $g^0$ we are looking for, belong to $\cF \otimes \M$ and for any couple $(f,g) \in \cF \otimes \M$,  equation \eqref{eq:model} describes the law of one continuous curve. Its law is denoted $\PP_{f,g}$ and possesses a density  $p_{f,g}$ with respect to the Wiener measure on the sample space. 
 Since $f^0$ and $g^0$ are unknown, $\PP_{f^0,g^0}$ is also unavailable but belongs to a set $\cP$ of probability measure over the sample space. This set $\cP$ is the set of all possible measures described by \eqref{eq:model} when $(f,g)$ varies into $\cF \otimes \M$.

 Given some prior distribution $\Pi_n$ on $\cP$ (generally defined through a prior on $\cF\otimes \M$), Bayesian procedures are generally built using the posterior distribution defined by
 $$
 \Pi_n\left(B|Y_1,\ldots, Y_n\right)  = \frac{\int_{B} \prod_{j=1}^n p(Y_j) d\Pi_n(p)}{\int_{\cP} \prod_{j=1}^n p(Y_j) d\Pi_n(p)},
 $$
 which is a random measure on $\cP$ that depends on the observations $Y_1, \ldots, Y_n $. For instance, Bayesian estimators can be obtained using the mode, the mean or the median of the posterior distribution. 
This is exactly the approach adopted by \cite{AKT} which is mainly dedicated to compute such a posterior mean in a parametric setting with a stochastic EM algorithm.
 
 The posterior distribution  is then said \emph{consistent} if it concentrates to arbitrarily small neighbourhoods  of $\PP_{f^0,g^0}$ in $\cP$ with a probability tending to $1$ when $n$ grows to $+\infty$. 
One \emph{frequentist} property of such a posterior distribution describes the contraction rate of such neighbourhoods meanwhile still capturing most of posterior mass.
According to equation~\eqref{eq:model}, we thus tackle such a Bayesian consistency and compute such convergence rates in the frequentist paradigm. Of course, these properties will highly depend on the metric structure of the sets $\cP$ and $\cF$.

\paragraph{Functional setting and Fourier analysis}
Without loss of generality, the function $f^0$ is assumed to be periodic with period $1$ and to belong to a subset $\cF$ of $L^2_\C([0,1])$, the space of squared integrable functions on $[0,1]$ endowed with the euclidean norm $\|h\|:=\int_0^1 |h(s)|^2 ds$. 
Moreover, each element $h\in L^2_\C([0,1])$ may naturally be extended to a periodic function on $\R$ of period $1$. Since we will intensively use some Fourier analysis in the sequel, let us first recall some notations: $\i$ will stand for the complex number such that $\mathfrak i^2=-1$. The Fourier coefficients of $h$ are denoted
\begin{equation}\label{eq:fourier}
\theta_{\ell}(h) := \int_{0}^1 e^{- \i 2 \pi  \ell t} h(t) dt.
\end{equation}
All along the paper, we will often use the parametrisation of any element of $h \in L_\C^2([0,1])$ through its Fourier expansion and  will simply use the notation $(\theta_{\ell})_{\ell \in \Z}$ instead of $(\theta_{\ell}(h))_{\ell \in \Z}$.

Our work is dedicated to the analysis of SIM when  $\cF$  models smooth functions of $[0,1]$. Hence, natural subspaces of $L^2_\C([0,1])$ are Sobolev spaces $\cH_s$ with a smoothness parameter $s$:
$$
\cH_s
:=\left\{f\in L^2_\C([0,1]) \quad \vert \quad \sum_{\ell \in \Z} (1+|\ell|^{2s}) |\theta_\ell(f)|^2 
< + \infty \right\}.
$$
In the sequel,  we aim to find prior on $\cP$ that reaches good frequentist properties, and if possible adaptive with the smoothness parameter $s$ since this parameter is generally unknown.
We will consider only some regular cases when $s \geq 1$, the quantity $\sum_{\ell} \ell^2 |\theta_\ell|^2$ is thus bounded and we denote the Sobolev norm
$$
\|\theta\|_{\cH_1} := \sqrt{ \sum_{\ell\in \Z} \ell^2 |\theta_\ell|^2}.
$$
It will also be useful to consider in some cases Fourier "thresholded" elements of $\cH_s$. Hence, we set for any integer $\ell$ (which is the frequency threshold)
$$
\cH^\ell := \left\{f\in L^2_\C([0,1])
\quad \vert \quad \forall |k|>\ell \quad \theta_k(f) = 0 \right\}.
$$

\paragraph{Mixture  model}
According to equation \eqref{eq:model}, we can write in the Fourier domain that
\begin{equation*}
\forall \ell \in \Z \quad \forall j \in \{1 \ldots n\}  \qquad 
\theta_{\ell}(Y_j) = \theta^0_{\ell} e^{- \i 2 \pi j \tau_j} + \xi_{\ell,j},
\end{equation*}
where $\theta^0: =(\theta^0_{\ell})_{\ell \in \Z}$ denotes the true unknown Fourier coefficients of $f^0$. Owing to the white noise model, the variables $(\xi_{\ell,j})_{\ell,j}$ are independent standard (complex) Gaussian random variables: $\xi_{\ell,j} \sim_{i.i.d.} \cN_{\C}(0,1), \forall \ell,j$. 

For sake of simplicity,  $\gamma$ will refer to $\gamma(z) := \pi^{-1} e^{-|z|^2}, \forall z \in \C$, the density of the standard complex Gaussian centered distribution $\cN_{\C}(0,1)$, and $\gamma_{\mu}(.) := \gamma(.-\mu)$ is the density of the standard complex Gaussian with mean $\mu$.
We keep also the same notation for $p$ dimensional complex Gaussian densities 
$\gamma(z) := \pi^{-p} e^{-\|z\|^2}, \forall z \in \C^p$, where $\|z\|$ is the euclidean $p$ dimensional norm of the complex vector $z$.

For any frequence $\ell$, equation \eqref{eq:model} implies that $\theta_\ell(Y)$ follows a mixture of complex Gaussian standard variables with mean $\theta^0_{\ell} e^{- \i 2 \pi \ell \varphi}, \varphi \in [0,1]$:
$$
\theta_\ell(Y) \sim \int_{0}^1 \gamma_{\theta^0_{\ell} e^{- \i 2 \pi \ell \varphi}}(\cdot) dg(\varphi).
$$
In the sequel, for any phase $\varphi \in [0,1]$ sampled according to any distribution $g$, and for any $\theta \in \ell^2(\Z)$, $\theta \bullet \varphi$ will denote the element of $\ell^2(\Z)$ given by
$$
\forall \ell \in \Z \qquad (\theta \bullet \varphi)_{\ell}: = \theta_{\ell} e^{- \i 2 \pi \ell \varphi}.
$$
When $\theta$ is a complex vector, 
for instance $\theta = (\theta_{-\ell}, \ldots, \theta_\ell)$, we keep the same notation  $\theta \bullet \varphi := (\theta_{-\ell} e^{\i 2 \pi \ell \varphi}, \ldots, \theta_0, \theta_1 e^{-\i 2 \pi \varphi}, \ldots, \theta_\ell e^{-\i 2 \pi \ell \varphi})$ to refer to the $2\ell+1$ dimensional vector.
It  corresponds to a rotation of each coefficient $\theta_{\ell}$  around the origin with an angle $2 \pi \ell \varphi$. 
According to this notation, the law of the infinite series (of Fourier coefficients of $Y$) can thus be rewritten as
$$
\theta(Y) \sim \int_{0}^1 \gamma_{\theta^0 \bullet \varphi}(.) dg(\varphi).
$$
One should remark the important fact that from one frequency to another, the rotations used to build $\theta(Y)$ are not independent, which traduces the fact that the coefficients $\left(\theta_{\ell}(Y)\right)_\ell$ are highly correlated.

\subsection{Notations on Mixture models}

Our study will intensively use some classical tools of mixture models,  see for instance the papers of \cite{GvdW01} or \cite{GW}. We thus choose to keep some notations already used in such works. 

For any vector $\theta \in \ell^2_\C(\Z)$ corresponds a function $f \in L^2([0,1])$ according to equation \eqref{eq:fourier} and for any measure $g \in \M$, $\PP_{\theta,g}$ will refer to the law of the vector of $\ell^2(\Z)$ described by the location mixture of Gaussian variables:
$$
\PP_{\theta,g}: = \int_{0}^1 \gamma_{\theta \bullet \varphi}(.) dg(\varphi).
$$
This mixture model is of infinite dimension since $\theta$ belongs to $\ell^2(\Z)$. Following an obvious notation shortcut, $\PP_{f,g}$ will be its equivalent for the functional law on curves derived from $\PP_{\theta,g}$. When $\theta$ is of finite length $k$, $p_{\theta,g}$ will be the density with respect to the Lebesgue measure on $\C^k$ of the law $\PP_{\theta,g}$:
$$
\forall z \in \C^k \qquad 
p_{\theta,g}(z): = \int_{0}^1 \gamma (z - \theta \bullet \varphi) dg(\varphi).
$$

We also use standard objects such as the Hellinger distance $d_H$ between probability measures and the Total Variation distance $d_{TV}$, as well as covering numbers of metric spaces such as $D(\epsilon,\cP,d)$. These objects are precisely described in Appendix \ref{section:appendix_proba}.

\paragraph{Bayesian frequentist consistency rate}

In our setting,  $d$ is chosen according to one of the metric introduced above 
($d_H$ or $d_{TV}$)
on the set  
$$\cP:= \left\{ \PP_{f,g}  \vert   (f,g) \in \cH_s
\otimes \M \right\}.$$
We can now remind Theorem 2.1 of \cite{GGvdW00} which will be useful for our purpose. 
\begin{theo}[Posterior consistency and convergence rate, \cite{GGvdW00}]\label{theo:posterior}
Assume that  a sequence $(\epsilon_n)_n$ with $\epsilon_n \rightarrow 0$ and $n \epsilon_n^2  \rightarrow + \infty$, a constant $C>0$, and a sequence of sets $\cP_n \subset \cP$ satisfy

\begin{equation}\label{eq:bound_covering}
\log D(\epsilon_n,\cP_n,d) \leq n \epsilon_n^2
\end{equation}
\begin{equation}\label{eq:bound_sieve}
\Pi_n \left( \cP \setminus \cP_n \right) \leq e^{- n \epsilon_n^2 (C+4)}
\end{equation}
\begin{equation}\label{eq:bound_neighbourhood}
\Pi_n \left( \PP_{f,g} \in \cP \vert d_{KL} (\PP_{f^0,g^0}, \PP_{f,g})
 \leq \epsilon_n^2 , V (\PP_{f^0,g^0}, \PP_{f,g}) \leq \epsilon_n^2 \right) \geq e^{-  n \epsilon_n^2 C}.
\end{equation}
Then there exists a sufficiently large $M$ such that
 $$\Pi_n\left( \PP_{f,g} : d(\PP_{f^0,g^0}, \PP_{f,g})  \geq M \epsilon_n | Y_1, \ldots Y_n \right) \longrightarrow 0$$ 
 in $\PP_{f^0,g^0}$ probability as $n\longrightarrow + \infty$.
\end{theo}
The posterior concentration rate obtained in the above result is $\epsilon_n$. The growing set $\cP_n$ is referred to as a Sieve over $\cP$. 
Generally,  this rate $\epsilon_n$ can be compared to the classical frequentist benchmark: for instance \cite{GGvdW00} obtained for the Log Spline model a contraction rate $\epsilon_n= n^{-s/(2s+1)}$ when the unknown underlying density belongs to an Hˆlder class $\cC^s([0,1])$, 
and this rate is known to be the optimal one (in the sense that it is the minimax one) in the frequentist paradigm over Hˆlder densities of regularity $s$ (see \cite{ibragimov_book}). 
Similarly, the recent work of \cite{RR} considers the situation of density estimation for infinite dimensional exponential families and reaches also contraction rates close or equal to the known optimal frequentist one.

\subsection{Bayesian prior and posterior concentration in the randomly shifted curves model}\label{section:prior}

We detail here the Bayesian prior $\Pi_n$ on $\cP$ used to obtain a polynomial concentration rate. Note that such prior will be in our work independent on the unknown smoothness  parameter $s$. 
As pointed in the paragraph above, it is sufficient to define some prior on the space $\cH_s \otimes \M$ since equation \eqref{eq:model} will then transport this prior to a law $\Pi_n$ on $\cP$. The two parameters $f$ and $g$ are picked independently at random following the next prior distributions.

\paragraph{Prior on $f$} The prior on $f$ is slightly adapted from \cite{RR}. It is defined on $\cH_s$ through 
$$
\pi := \sum_{\ell \geq 1} \lambda(\ell) \pi_\ell.
$$
Given any integer $\ell$, the idea is to decide to randomly switch on with probability $\lambda(\ell)$ all the Fourier frequencies from $-\ell$ to $+ \ell$. Then,  $\pi_{\ell}$ is  a distribution defined on $\ell^2(\Z)$  such that $\pi_{\ell}:=\otimes_{k \in \Z} \pi_{\ell}^k$ and
$$
\forall k \in \Z \qquad \pi^k_{\ell}= 1_{|k| > \ell} \delta_{0}+ 1_{|k| \leq \ell} \cN_\C(0,\xikn^2). 
$$
The randomisation of selected frequencies is done using $\lambda$, a probability distribution on $\N^{\star}$ which satisfies for $\rho \in (1,2)$:
\begin{equation*}
 \exists (c_1,c_2) \in \R_+ \quad \forall \ell \in \N^{\star}  \qquad 
e^{-c_1 \ell^2\log^\rho \ell} \lesssim \lambda(\ell) \lesssim e^{-c_2 \ell^2 \log^\rho \ell}.
\end{equation*}

The prior $\pi$ depends on the variance of the Gaussian laws  $\xikn$ used to sample the Fourier coefficients. 
In the sequel, we use a variance that depends on $n$ according to 
\begin{equation}\label{eq:variance_prior}
\xi_n^2 := n^{-\mu_s} (\log n)^{-\zeta},
\end{equation}
where $\mu_s$ and $\zeta$ are parameters that may depend on $s$ (non adaptive prior) or not (adaptive prior).

\paragraph{Prior on $g$} As our model  does not seem so far from a mixture Gaussian model, a natural prior on $g$ is built according to a Dirichlet process following the ideas of \cite{GvdW01}. Given any finite base measure $\alpha$ that has a positive continuous density on $[0,1]$ w.r.t. the Lebesgue measure, the Dirichlet process $D_{\alpha}$ generates a random probability measure $g$ on $[0,1]$. For any finite partition $(A_1, \ldots, A_k)$ of $[0,1]$, the probability vector $(g(A_1), \ldots, g(A_k))$ on the $k$-dimensional simplex has a Dirichlet distribution $Dir(\alpha(A_1), \ldots, \alpha(A_k))$. Such process may be built according to the Stick-Breaking construction (see for instance \cite{F73}).

\subsection{Main result}

Using the prior defined above, we obtain the following theorem on the randomly SIM.
\begin{theo}\label{theo:posterior_shift}
Assume that 
$f^0 \in \cH_s$ with $s \geq 1$, then the values
 $\mu_s=2/(2s+2)$ and $\zeta=0$ in the definition of $\xikn$  yield a non adaptive prior  such that
$$
\Pi_n \left\{ \PP_{f,g} \text{ s.t. } d_H(\PP_{f,g} ,\PP_{f^0,g^0}) \leq M \epsilon_n \vert Y_1, \ldots Y_n \right\} = 1 + \cO_{\PP_{f^0,g^0}}(1)
$$
when $n \longrightarrow + \infty$, for a sufficiently large constant $M$ such that. Moreover, the contraction rate $\epsilon_n$ is given by
$$
\epsilon_n = n^{- s / (2 s+2)} \log n.
$$
The values $\mu=1/4$  and $\zeta=3/2$ yield the contraction rate
$$
\Pi_n \left\{ \PP_{f,g} \text{ s.t. } d_H(\PP_{f,g} ,\PP_{f^0,g^0}) \leq M \epsilon_n \vert Y_1, \ldots Y_n \right\} = 1 + \cO_{\PP_{f^0,g^0}}(1)
$$
for a sufficiently large constant $M$, when $n \longrightarrow + \infty$ with
$$
\epsilon_n =\left\lbrace
\begin{array}{ll}
n^{-s/(2s+2)} \log n &\quad \text{if} \quad s \in [1,3] \\
n^{-3/8} \log n &\quad \text{if} \quad s \geq 3.
\end{array}\right.
$$
\end{theo}

Let us briefly comment this result. It first describes the posterior concentration around some neighbourhood of the true  law $\PP_{f^0,g^0}$ within a polynomial rate. Our prior is adaptive with the regularity $s$ as soon as $s\in[1,3]$ setting $\xi_n^2 = n^{-1/4} (\log n)^{-3/2}$. For this range of $s$, the convergence rate is $n^{-s/(2s+2)}$ up to a logarithmic term. To the best of our knowledge, the minimax frequentist rate is unknown for the problem on recovering $\PP_{f^0,g^0}$ when both $f^0$ and $g^0$ are unknown. An interpretation of such polynomial rate is rather difficult to provide. It may be interpreted as $-s/(2s+d)$ where $d$ is the number of dimension to estimate in the model ($f^0$ and $g^0$).
When $s$ becomes larger than $3$, the rate of Theorem \eqref{theo:posterior_shift} is "blocked" to $3/8$ (which corresponds to $s/(2s+2)$ when $s=3$) and does not match with $s/(2s+2)$. This difficulty is mainly due to the important condition $w_{\epsilon}^2 \lesssim l_{\epsilon}$ in Theorem \ref{theo:recouvrement}.

At last, the non adaptive prior based on $\xi_n^2 = n^{-2/(2s+2)}$ recovers the good rate $-s/(2s+2)$ for all $s$ larger than $1$.

The former result establish a result on the law $\PP_{f,g} \in \cP$.
It is also possible to derive a second result on the objects $f \in \cH_s$ themselves. This result is studied in \cite{BG_2} and provides a somewhat quite weak result on the posterior convergence towards the true objects $f^0$ and $g^0$. 

\section{Metric description of the model\label{sec:metric}}

We aim to check conditions (\ref{eq:bound_sieve}) and (\ref{eq:bound_neighbourhood}) and then apply  Theorem \ref{theo:posterior}. In this view, we first define in section \ref{sec:entropy} a sieve $\cP_{\ell_{\epsilon},w_{\epsilon}}$, and our goal is to find some optimal calibration of $\epsilon$, $l_\epsilon$ and $w_{\epsilon}$ with respect to $n$.
We thus need to find a lower bound of the prior mass around some Kullback-Leibler neighbourhood of $\PP_{f^0,g^0} \in \cP$. These sets are defined as
\begin{equation*}
\cV_{\epsilon_n}(\PP_{f^0,g^0},d_{KL}) =  \left\{ \PP_{f,g} \in \cP \vert d_{KL}(\PP_{f^0,g^0},\PP_{f,g}) \leq \epsilon_n^2 , V(\PP_{f^0,g^0},\PP_{f,g}) \leq\epsilon_n^2\right\}. 
\end{equation*}
This will be done indeed considering Hellinger neighbourhoods instead of  Kullback-Leibler ones. 
A link between these two kinds of neighbourhood is given in section \ref{sec:link}. In section \ref{sec:Hellinger}, we work with the Hellinger neighbourhoods to exhibit some admissible sizes for $\epsilon_n$, $\ell_n$ and $w_n$. 
At last, we prove Theorem \ref{theo:posterior_shift} in section \ref{sec:core_proof}.

In all this section, we delay most technical proofs to the Appendix.

\subsection{Entropy estimates} \label{sec:entropy}

We first establish some useful results on the complexity of our model $\PP_{f,g}$ when $f\in\cH_s$ and $g \in \M$ in various situations ($f$ known, unknown, parametric  or not).

\subsubsection[Case of known f]{Case of known $f$}


We first give some useful results when $f$ is known and belongs to a finite dimensional vector space (the number of active Fourier coefficients is restricted to $[-\ell,\ell]$ for a given $\ell$). Then $\ell$ will be allowed to grow with $n$ and depend on a  parameter $\epsilon$ introduced below.
 Hence, $f$ is described by the parameter $\theta = (\theta_{-\ell}, \ldots, \theta_0, \ldots, \theta_\ell)$, and we define the set of all possible Gaussian measures
 $$
 \cA_{\theta} := \left\{ \gamma_{\theta \bullet \varphi}, \varphi \in [0,1]\right\}.
 $$
  Following the arguments of \cite{GW}, it is possible to establish the following preliminary result.
\begin{proposition}\label{prop:atheta}
For any sequence $\theta \in \C^{2\ell+1}$, one has
$$
 N_{[]}(\epsilon,\cA_{\theta},d_H) \leq \frac{4 \pi \sqrt{2(2\ell+1)} \|\theta\|_{\cH_1}}{\epsilon} (1+o(1)),
$$
where $o(1)$ goes to zero independently on $\ell$ and $\theta$ as $\epsilon\rightarrow 0$, and
$$
\log N(\epsilon,\cA_{\theta},d_H) \lesssim \log \ell + \log \|\theta\|_{\cH_1} + \log \frac{1}{\epsilon}.
$$
\end{proposition}

Assume now that $g$ possesses a finite number of $k$ points in its support, one can deduce from the proposition above a simple corollary that exploits the complexity of the simplex of dimension $k-1$ (see for instance the proof of Lemma 2 in \cite{GW}).

\begin{proposition}\label{prop:Mktheta}
Assume that $f$ is parametric and known ($\theta \in \C^{2\ell+1}$) and define
\[ \cM^k_{\theta} := \left\{ \sum_{i=1}^k g(\varphi_i) \gamma_{\theta\bullet\varphi_i} : \varphi_i \in [0,1], \, g(\varphi_i) \geq 0,   \forall i \in \llbracket  1,k  
\rrbracket\, \text{and} \,  \sum_{i=1}^k g(\varphi_i) = 1\right\} \]
for a number of components $k$ that may depend on $\epsilon$ (as $\ell$ does). Then
$$
H_{[]}(\epsilon,\cM^k_{\theta},d_H) \lesssim k \left(\log \ell + \log \|\theta\|_{\cH_1} + \log \frac{1}{\epsilon}\right).
$$
\end{proposition}


We then naturally provide a description of the situation when $f$ is known and parametrized by an infinite sequence $\theta \in \ell^2(\Z)$. According to the previous computations, and using a truncation argument at frequency $\ell_\epsilon=\epsilon^{-1/s}$ in the Sobolev space $\cH_s
$, one can show the following result.

\begin{coro} \label{corol:Mktheta}
Assume $f \in \cH_s$ known for $s\geq 1$ ($\theta :=\theta(f)$ such that $\sum_{j\in\Z} |\theta_j|^{2} |j|^{2s} < +\infty$), using the same set $\cA_{\theta}$ as in Proposition \ref{prop:atheta} with $\ell_\epsilon=\epsilon^{-1/s}$, then
$$
H_{[]} (\epsilon,\cA_{\theta},d_H) \lesssim  \frac{s+1}{s}\log \frac{1}{\epsilon} + \log \|\theta\|_{\cH_1}.
$$
Similarly, one also has
$$
H_{[]}(\epsilon,\cM^k_{\theta},d_H) \lesssim k \left(\frac{ s+1}{s}\log \frac{1}{\epsilon} + \log \|\theta\|_{\cH_1}\right).
$$
\end{coro}

The next step is to consider a continuous mixture for $g$, which is the more natural case. For $f$ known, let
\[ \cP_{f} := \left\{ \PP_{f,g} \, \vert \, g \in \M\right\}. \]
Once again, we will only consider functions $f$ with null Fourier coefficients of order higher than $\ell_\epsilon$. For sake of simplicity, we will omit the dependence on $\epsilon$ with the notation $\ell$.

It would be quite tempting to use the results of \cite{GvdW01} to bound the bracketing entropy of $\cP_{f}$, but indeed as pointed by \cite{MM11} applying directly the bounds obtained in Lemma 3.1 and Lemma 3.2 of \cite{GvdW01} to our setting yields a too weak result: the size of the upper bound on $H_{[]}(\epsilon,\cP_{f},d_H)$ will have a too strong dependency on $\ell$. 
By the way, we have to carefully adapt the approach of \cite{GvdW01} to obtain a sufficiently sharp upper bound of the entropy of $\cP_{f}$. Such bound is given in the next result, in which we provide a majorization of the entropy with respect to the Total Variation distance which is easier to handle here. 
Note that all the previous results are still true if we use $d_H$ instead of $d_{TV}$ since \eqref{eq:dvt_dh} also permits to retrieve entropy bounds for $d_H$ from entropy bounds for $d_{TV}$.

\begin{proposition} \label{prop:Pf}
 Let $\epsilon>0$ and $s > 0$, if $\log \tfrac{1}{\epsilon} \lesssim \ell$ and $f \in \cH^\ell$ is such that $\|\theta\|^2 \lesssim 2\ell+1$, then
 \[ \log N (\epsilon,\cP_{f} ,d_{TV}) \lesssim \ell^2 \left(\log \frac{1}{\epsilon} + \log \|\theta\|_{\cH_1}\right).  \]
 If furthermore $w  \lesssim \sqrt{2\ell+1}$ then
 \[ \sup_{f \in \cH^\ell : \|\theta(f)\| \leq w} \log N (\epsilon,\cP_{f} ,d_{TV}) \lesssim   \ell^2 \left(\log \frac{1}{\epsilon} + \log \ell\right).\]
\end{proposition}

The second inequality opens the way for the case of unknown $f$ given below. It is possible since in the first inequality we have carefully expressed the dependency on $f$ and $\ell$.

 The method to build an $\epsilon$-covering of $\cP_f$ follows two natural steps:
 \begin{itemize}
 \item 
  approximate any mixture $g$ by a finite one $\tilde{g}$ such that $$d_{TV}(\PP_{\theta,g},\PP_{\theta,\tilde{g}}) \leq\epsilon/2,$$ with a number of components of the finite mixture $\tilde{g}$ uniformly bounded in $g$ (depending on $f$ and $\epsilon$);
  \item    use Proposition \ref{prop:Mktheta} for the finite mixture to well approximate $\PP_{\theta,\tilde{g}}$. 
 \end{itemize}
 The proof itself is delayed to the Appendix.

\subsubsection[Case of unknown f]{Case of unknown $f$}

We now describe the picture when $f$ is unknown, which is the main objective of this paper. We assume that $f$ belongs to $\cH_s$. In order to bound the bracketing entropy, we define a sieve over 
$\cH_s$ which depends on a frequency cut-off  $\ell$  and a size parameter $w$. We then get
$$
\cP_{\ell,w} := \left\{ \PP_{f,g} \, \vert \, f \in \cH_s^{\ell}, \|\theta(f)\| \leq w, g \in \M\right\}.
$$

\begin{theo}\label{theo:recouvrement}
Let be given $\epsilon>0$ small enough, and 
assume that $\ell_{\epsilon}$ and $w_{\epsilon}$ are such that $ \log \frac{1}{\epsilon} \lesssim l_{\epsilon}$ and $w_{\epsilon} \lesssim \sqrt{\ell_\epsilon}$,  then
$$
\log N (\epsilon,\cP_{\ell_{\epsilon},w_\epsilon} ,d_{TV}) \lesssim  l_\epsilon^2 \left(\log \frac{1}{\epsilon} + \log \ell_\epsilon\right).
$$
\end{theo}
The proof of Theorem \ref{theo:recouvrement} is based on two simple results. The first one is the Girsanov formula obtained by \cite{BG10} in appendix A.2.2 (in the case of known $g$): it can be extended to the situation of unknown $g$ and complex trajectories as in \eqref{eq:model}, which leads to
\begin{equation} \label{eq:Girsanov}
 \frac{d\PP_{f,g}}{d\PP_{f^0,g^0}}(Y) = \frac{\int_{0}^1 exp \left( 2\re\langle f^{- \alpha_1}, d Y \rangle - \|f^{-\alpha_1}\|^2\right) d g(\alpha_1) }{\int_{0}^1 exp \left(2 \re \langle f^{0, - \alpha_2}, d Y \rangle - \|f^{0, -\alpha_2}\|^2\right) d g^0(\alpha_2) },
\end{equation}
for any measurable trajectory $Y$.

The second result is given in the following lemma.
\begin{lemma} \label{lemma:dVT_sur_f}
 Let $f$ and $\tilde{f}$ be any functions in $L^2_\C([0, 1])$, $g$ be any shift distribution in $\M$, then
 \begin{equation*}
  d_{TV}(\PP_{f,g},\PP_{\tilde{f},g}) \leq \frac{\|f-\tilde{f}\|}{\sqrt{2}}.
 \end{equation*}
\end{lemma}

\begin{proof}[Proof of Theorem \ref{theo:recouvrement}]
The idea of the demonstration is to build a $\epsilon$-covering of $\cP_{\ell,w}$ with $\epsilon/2$-coverings for $f$ and $g$. First, let $\PP_{f,g}$ and  $\PP_{\tilde{f},\tilde{g}}$ two elements of $\cP_{l,w}$ and remark that by the triangle inequality
$$
d_{TV}(\PP_{f,g},\PP_{\tilde{f},\tilde{g}}) \leq d_{TV}(\PP_{f,g},\PP_{\tilde{f},g}) + d_{TV}(\PP_{\tilde{f},g},\PP_{\tilde{f},\tilde{g}}).
$$
We will look for a covering method that will use the inequality above and a tensorial argument, it requires to bound both terms. The majorization of the first one comes from Lemma \ref{lemma:dVT_sur_f}. The second term is handled uniformly in $\tilde{f}$ by Proposition \ref{prop:Pf}. 

 Now, we build $\epsilon/2$-coverings of $\PP_{f,g}$ for fixed $g$ from an $\epsilon/\sqrt{2}$-covering of $f$ for the $L^2$-norm:
 \[ \log N \left(\epsilon/\sqrt{2}, \left\{ f \in \cH_s^{\ell_\epsilon}, \|\theta(f)\| \leq w_\epsilon\right\}, \|\cdot\|\right) \lesssim \ell_\epsilon \log \frac{w_\epsilon}{\epsilon} = o\left( \ell_\epsilon^2 \log \frac{1}{\epsilon}\right). \]
\end{proof}
\noindent
According to inequality \eqref{eq:dvt_dh} and since $\log \frac{1}{\epsilon^2} \lesssim \log \frac{1}{\epsilon}$, we can easily deduce the next corollary.
\begin{coro}\label{coro:hellinger}
Let be given $\epsilon>0$ small enough, and assume $\log \frac{1}{\epsilon} \lesssim \ell_{\epsilon}$ and $w_\epsilon \leq \sqrt{2\ell_\epsilon+1}$, then
$$
\log N (\epsilon,\cP_{\ell_\epsilon,w_\epsilon},d_H) \lesssim \ell_{\epsilon}^2 \left(\log \frac{1}{\epsilon} + \log \ell_\epsilon \right).
$$
\end{coro}

\begin{remark}
i) Even if the model studied here is a very special case of Gaussian mixture models, one may think that such kind of results may help the analysis of more general mixture cases within a  growing dimension setting.

ii) In our case, we will use a much higher choice of $l_{\epsilon}$ than $\log \frac{1}{\epsilon}$. This choice will be fixed in section \ref{sec:core_proof}.
\end{remark}

\subsection{Link between Kullback-Leibler and Hellinger neighbourhoods}\label{sec:link}

We first recall a useful result of Wong \& Shen given as Theorem 5 in \cite{WS95}. It enables to handle Hellinger neighbourhoods instead of $\cV_{\epsilon_n}(\PP_{f^0,g^0},d_{KL})$, which is generally easier for mixture models.

\begin{theo}[Wong \& Shen]\label{theo:wong_shen}
Let $\mu$ and $\nu$ be two measures such that $\mu$ is a.c. with respect to $\nu$ with a density $q = d \mu/d\nu$. Assume that $d_H(\mu,\nu)^2 = \int [\sqrt{q} -1]^2 d  \nu\leq \epsilon^2$ and that there exists $\delta \in (0,1]$ such that 
\begin{equation}\label{eq:condition_moment}
M_\delta^2 
:=\int_{q\geq e^{1/\delta}} q^{\delta+1} d\nu < \infty.
\end{equation}
Then, for $\epsilon$ small enough, there exists a universal constant $C$ large enough   such that
$$
d_{KL}(\mu,\nu)  = \int q \log q d \nu \leq C \log(M_\delta)\epsilon^2 \log \frac{1}{\epsilon},
$$
and
$$
V(\mu,\nu) \leq \int q  \log^2 q d \nu \leq C \log(M_\delta)^2 \epsilon^2 \left[\log \frac{1}{\epsilon}\right]^2. 
$$
\end{theo}
Hence, Hellinger neighbourhoods are almost Kullback-Leibler ones (up to some logarithm terms) provided that a sufficiently large moment exists for $q$ ($q \log q$ is killed by $q^{1+\delta}$ for large values of $q$ and a second order expansion of $q \log q - q + 1$ around $1$ yields a term similar to $[\sqrt{q}-1]^2$). 
Next proposition shows that condition (\ref{eq:condition_moment}) is satisfied in our SIM.

\begin{proposition}\label{prop:appli_wong_shen}
 For any $\PP_{f^0,g^0} \in \cP$, and for any $f\in \cH_s$ such that $\|f\|\leq 2 \|f^0\|$, and any $g \in \M$, define $q=\frac{d\PP_{f^0,g^0}}{d\PP_{f,g}}$. There exists $\delta \in (0,1]$ such that the constant defined in equation \eqref{eq:condition_moment} $M_\delta^2$ is uniformly bounded with respect to $f$.
\end{proposition}

 \subsection{Hellinger neighbourhoods}\label{sec:Hellinger}
Proposition \ref{prop:appli_wong_shen} will allow to use  Theorem \ref{theo:wong_shen}, thus we now aim to find a lower bound on Hellinger neighbourhood of $\PP_{f^0,g^0}$. Consider a frequency cut-off $\ell_n$ that will be fixed later. For any $f  \in \cH_s^{\ell_n}$ and $g\in \M$, remind that we denote $\theta:=\theta(f)$ as well as $\theta^0 = \theta(f^0)$. We define $f^0_{\ell_n}$ the $L^2$ projection of $f^0$ on the subspace $\cH_s^{\ell_n}$.
 
 For sake of simplicity, $\EE_0F(Y)$ will refer to the expectation of a function $F$ of the trajectory $Y$ when $Y$ follows $\PP_{f^0,g^0}$.
The triangle inequality applied to the Hellinger distance shows that
\[ d_H(\PP_{f^0,g^0},\PP_{f,g}) \leq \overbrace{d_H(\PP_{f^{0},g^0},\PP_{f^{0}_{\ell_n},g^0})}^{(E_1)} + \overbrace{d_H(\PP_{f^{0}_{\ell_n},g^0},\PP_{f^{0}_{\ell_n},g})}^{(E_2)} + 
\overbrace{d_H(\PP_{f^{0}_{\ell_n},g},\PP_{f,g})}^{(E_3)}. \]
In the sequel, we will provide sufficiently sharp upper bound on $(E_1)$, $(E_2)$, $(E_3)$ so that we will be able to find a suitable lower bound of the prior mass of Hellinger neighbourhoods.

\paragraph{Upper bound of $(E_1)$} We first bound $(E_1)$ using $d_H^2 \leq d_{KL}$ with the Girsanov formula \eqref{eq:Girsanov}
\begin{align*}
(E_1) 
 &\leq \sqrt{d_{KL}(\PP_{f^0,g^0}, \PP_{f^{0}_{\ell_n},g^0})}\\
 &= \left(\EE_0 \left[ - \log \frac{\int_0^1 \exp\left(2\re \langle f^{0,-\alpha}_{ \ell_n}, d Y\rangle - \|f^{0}_{\ell_n}\|^2\right) d g^0(\alpha)}{\int_0^1 \exp\left(2\re \langle f^{0,-\alpha}, d Y\rangle - \|f^0\|^2\right) d g^0(\alpha)} \right]\right) ^{1/2}\\
 & := (\tilde{E_1})\\
 \end{align*}
We now obtain the upper bound of $(E_1)$  according to the next proposition.

\begin{proposition}\label{prop:E1} Assume that $Y \sim \PP_{f^0,g^0}$ and $f^0\in \cH_s$, then
 $$ (E_1) \leq (\tilde{E_1}) \leq \sqrt{2} \|f^0-f^{0}_{\ell_n}\|\leq \sqrt{2} \|f^0\|_{\cH_1} \ell_n^{-s}.$$
\end{proposition}

\paragraph{Upper bound of $(E_3)$}
 
 We will be interested in the Hellinger distance when $f^0_{\ell_n}$ is close to $f$, and the dimension $\ell_n$ grows up to $+\infty$ (the mixture law on $[0,1]$ is the same for the two laws). The important fact will be its exclusive dependence with respect to the $L^2$ distance between $f^0_{\ell_n}$ and $f$. 
 This upper bound is given in the next proposition, whose proof is immediate from Lemma \ref{lemma:dVT_sur_f} and equation~\eqref{eq:dvt_dh}.

\begin{proposition}\label{prop:E3}
Assume that $f \in \cH_s^{\ell_n}$ and $g\in \M$, then
\[ d_H(\PP_{f^0_{\ell_n},g},\PP_{f,g}) \leq 2^{1/4} \sqrt{ \|f-f^0_{\ell_n}\|}. \]
\end{proposition}

\paragraph{Upper bound for $(E_2)$}
This term is clearly the more difficult to handle. We will obtain a convenient result using some elements obtained in Proposition \ref{prop:Pf}. For a given $\epsilon_n>0$, $\ell_n, f^0_{\ell_n} \in \cH_s^{\ell_n}$ and $g^0 \in \M$, we know that one may find a mixture model $\tilde{g}$ such that $d_H(\PP_{f^0_{\ell_n},g^0},\PP_{f^0_{\ell_n},\tilde{g}})< \epsilon_n$ and $\tilde{g}$ has $C \ell_n^2$ points of support in $[0,1]$ as soon as $\epsilon_n$ is small enough and $\log \frac{1}{\epsilon_n} \lesssim \ell_n$ (the condition $\|f^0_{\ell_n}\|^2\leq 2\ell_n+1$ is immediate since $f^0$ does not depend on $n$). The next step is to control the Hellinger distance $d_H(\PP_{f^0_{\ell_n},g},\PP_{f^0_{\ell_n},\tilde{g}})$ for $g\in\M$, and this can be done thanks to an adaptation in dimension $2 \ell_n+1$ of Lemma 5.1 of \cite{GvdW01}. 

\begin{lemma}\label{lemma:mixture_lemma}
Let be given $\tilde{g}$ a discrete mixture law whose support is of cardinal $J$ whose support points $(\varphi_j)_{j =1 \ldots J}$ are such that $\tilde{g}(\varphi_j)=p_j$ and $\eta$-separated, \textit{i.e.} $|\varphi_j - \varphi_i| \geq \eta, \forall i \neq j$, then  $\forall \check{g} \in \M$
$$
d_H^2(\PP_{f^0_{\ell_n},\tilde{g}},\PP_{f^0_{\ell_n},\check{g}})\leq \sqrt{\frac{\pi}{2}} \|f^0_{\ell_n}\|_{\cH_1} \eta + 2 \sum_{j=1}^J \left| \check{g}([\varphi_j-\eta/2,\varphi_j+\eta/2]) - \tilde{g}(\varphi_j)\right|.
$$
\end{lemma}
In \cite{BG_2}, we will show that it is possible to obtain a more general upper bound for the Hellinger distance between $\PP_{f^0_{\ell_n},\tilde{g}}$ and $\PP_{f^0_{\ell_n},g}$ which implies the Wasserstein distance $W_1(g,\tilde{g})$ between $g$ and $\tilde{g}$, but such upper bound is a little bit less powerful than the one given by the former lemma.
Note that Lemma \ref{lemma:mixture_lemma} needs a discrete mixture with $\eta$-separated support points. The following result permits to obtain such a mixture.

\begin{proposition} \label{prop:finite_mixture}
 Assume that $f^0 \in \cH_s$ for $s\geq 1$, $g^0\in\M$, and $\log \frac{1}{\epsilon_n} \lesssim \ell_n$. For any $\eta_n\leq\epsilon_n^2$, there exists a discrete distribution $\tilde{g}$ with in its support at most $J_n\lesssim \ell_n^2$ points denoted $(\psi_j)_{j=1 \ldots J_n}$, such that these points are $\eta_n$-separated, and
 \[ d_H(\PP_{f^{0}_{\ell_n},g^0},\PP_{f^{0}_{\ell_n},\tilde{g}})\leq\left(1+(8\pi)^{1/4}\|f^0_{\ell_n}\|_{\cH_1}^{1/2}\right) \epsilon_n. \]
 Furthermore, for any $g\in\M$,
 \begin{multline*} 
 d_H(\PP_{f^{0}_{\ell_n},g^0},\PP_{f^{0}_{\ell_n},g}) 
 \leq \left(1+(8\pi)^{1/4}\|f^0\|_{\cH_1}^{1/2}\right) \epsilon_n\\ + \sqrt{ \sqrt{\frac{\pi}{2}} \|f^0\|_{\cH_1} \eta_n + 2 \sum_{j=1}^{J_n} \left|g(\psi_j - \eta_n/2, \psi_j+\eta_n/2) - \tilde{g}(\psi_j)\right|}.
\end{multline*}
\end{proposition}

In \cite{BG_2} we obtain a more general upper bound for $(E_2)$, based on the Wasserstein distance. 
We could use it to retrieve Proposition \ref{prop:finite_mixture}, but it also leads to Hellinger neighbourhoods described in terms of the Total Variation distance from $g$ to $g^0$. This last distance is adapted to smooth densities $g$ but not to the ones considered here, when the prior distribution for $g$ is a Dirichlet process.

\paragraph{Description of a Hellinger neighbourhood}

We can now gather the upper bounds of $(E_1)$, $(E_2)$, and $(E_3)$ to get the following result.

\begin{proposition}\label{prop:Hellinger_neighbourhood}
 Assume that $f^0 \in \cH_s$ for $s\geq 1$ and $g^0\in\M$. Choose the threshold such as  $\epsilon_n^{-1/s} \lesssim \ell_n \lesssim \epsilon_n^{-1/s}$ and $\eta_n := \epsilon_n^2$, and consider the finite mixture $\tilde{g}$ provided by Proposition \ref{prop:finite_mixture}. Define 
 \begin{align*}
  \cG_{\epsilon_n} &:= \left\{ g \in \M :  \sum_{j=1}^{J_n} |g(\psi_j - \eta_n/2, \psi_j+\eta_n/2) - \tilde{g}(\psi_j)| \leq \epsilon_n^2 \right\}, \\
  \cF_{\epsilon_n} &:= \left\{ f \in \cH_s^{\ell_n} : \|f-f^0_{\ell_n}\| \leq \epsilon_n^2 \right\}.
 \end{align*}
 Then, there exists a constant $C_{0}$ depending only on $\|f^0\|_{\cH_1}$ such that 
 for any $g \in \cG_{\epsilon_n}$ and $f \in \cF_{\epsilon_n}$,
 $$ d_H\left(\PP_{f^0,g^0}, \PP_{f,g}\right) \leq C_{0} \epsilon_n. $$
\end{proposition}

\section{Proof of Theorem \ref{theo:posterior_shift}} \label{sec:proof_main}
We will prove this result using the "toolbox"  provided by  Theorem \ref{theo:posterior}. We thus check its applicability and consider each of its hypotheses.

\subsection{Checking the conditions of Theorem \ref{theo:posterior}}\label{sec:core_proof}

We first prove the minoration for the lower bound (\ref{eq:bound_neighbourhood}), necessary to apply Theorem \ref{theo:posterior}.

\begin{proposition}\label{prop:prior_mino}
Assume that  $f^0 \in \cH_s$ for $s\geq 1$ and $g^0\in\M$. 
For any sequence $(\epsilon_n)_{n \in \N}$ which converges to $0$ as $n \rightarrow + \infty$, 
and for the  prior defined in paragraph \ref{section:prior}, there exists a constant $c>0$ such that
\[ \Pi_n \left( \PP_{f,g} \in \cP : d_{KL} (\PP_{f,g}, \PP_{f^0,g^0}) \leq \epsilon_n^2 , V (\PP_{f,g}, \PP_{f^0,g^0}) \leq \epsilon_n^2 \right) \geq h_n,\] 
where 
$$
h_n := e^{-(c +o(1))\, \left[ \epsilon_n^{-2/s} \left(\log (1/\epsilon_n)\right)^{\rho+2/s} \vee \xikn^{-2} \right] }.
$$
\end{proposition}

Proposition \ref{prop:prior_mino} relies on Theorem \ref{theo:wong_shen}, which permits to use Hellinger neighbourhoods instead of $\cV_{\epsilon_n}(\PP_{f^0,g^0},d_{KL})$, 
and on Proposition \ref{prop:Hellinger_neighbourhood}, which describes suitable Hellinger neighbourhoods. To control their prior mass, we remind the following useful result appeared as Lemma 6.1 of \cite{GGvdW00}. This enables to find a lower bound of $\ell_1$-ball of radius $r$ under Dirichlet prior.

\begin{lemma}[\cite{GGvdW00}]\label{lemma:mino_simplex}
Let $r>0$ and $(X_1, \ldots, X_N)$ be distributed according to the Dirichlet distribution on the $\ell_1$ simplex of dimension $N-1$ with parameters $(m,\alpha_1, \ldots, \alpha_N)$. Assume that $\sum_{j} \alpha_j =m$ and $A r^{b} \leq \alpha_j \leq 1$ for some constants $A$ and $b$. Let $(x_1, \ldots, x_N)$ be any points on the $N$ simplex, there exists $c$ and $C$ that only depend on $A$ and $b$ such that if $r \leq 1/N$
$$
\Pr \left(\sum_{j=1}^N |X_j - x_j| \leq 2 r  \right) \geq C \exp \left(-c N \log \frac{1}{r} \right)
$$
\end{lemma}

In the proof of Proposition  \ref{prop:prior_mino} (delayed to the Appendix), one can see that we could obtain a suitable lower bound as soon as  $\lambda(\ell_n) \geq e^{- c \ell_n^2 \log \ell_n}$ for a constant $c$. Of course, a distribution $\lambda$ with some heavier tail would also suit here. However, such a heavier tail is not suitable for the control of the term $\Pi_n\left(\cP \setminus \cP_n\right)$ which 
is detailed in the next proposition.

\begin{proposition} \label{prop:prior_majo}
For any sequences  $k_n \mapsto + \infty$ and $\epsilon_n \mapsto 0$ as $n \mapsto + \infty$, 
define $w_n^2 = 4k_n+2$, then there exists a constant $c$ such that
$$
\Pi_n \left(\cP \setminus \cP_{k_n,w_n}\right)  \leq  e^{ 
- c [ k_n^2 \log^\rho(k_n) \wedge k_n \xikn^{-2}]},
$$
and
$$
\log D\left( \epsilon_n, \cP_{k_n,w_n},d_H \right) \lesssim  k_n^2 \left[ \log k_n + \log \frac{1}{\epsilon_n} \right].
$$
\end{proposition}

We are now able to conclude the proof of the posterior consistency.

\begin{proof}[Proof of Theorem \eqref{theo:posterior_shift}]
 Take $\epsilon_n := n^{-\alpha} (\log n)^{\kappa}$ and $k_n := n^{\beta} (\log n)^{\gamma}$. From our definition \eqref{eq:variance_prior}, we have also $\xikn^{-2} = n^{\mu_s} (\log n)^{\zeta}$, and we look for admissible values of $\alpha$, $\beta$, $\kappa$, $\gamma$, $\mu_s$, and $\zeta$ in order to satisfy \eqref{eq:bound_covering}, \eqref{eq:bound_sieve} and \eqref{eq:bound_neighbourhood}.

 Proposition \eqref{prop:prior_mino} imposes that in order to satisfy \eqref{eq:bound_neighbourhood}, we could check that
$$
\epsilon_n^{-2/s} \left (\log \frac{1}{\epsilon_n}\right)^{\rho+2/s} \vee n^{\mu_s} (\log n)^{\zeta} \ll n \epsilon_n^2 = n^{1-2\alpha} (\log n)^{2 \kappa}.
$$
This is true as soon as $\epsilon_n$ satisfies 
\begin{equation*}
\alpha \leq \frac{s}{2s+2} \qquad \text{and} \qquad 
\kappa > (\rho s+2)/(2s+2).
\end{equation*}
 Moreover, we obtain the first  condition on $\mu_s$: 
$\mu_s \leq 1 - 2\alpha$, and if $\mu_s = 1 - 2\alpha$ then $\zeta < 2 \kappa$.

Now, Proposition \eqref{prop:prior_majo} shows that \eqref{eq:bound_covering} is fulfilled provided that
\begin{equation}\label{eq:cond_kn}
k_n^2 \left[\log k_n+ \log \frac{1}{\epsilon_n} \right] \lesssim n \epsilon_n^2 =  n^{1-2\alpha} (\log n)^{2 \kappa}.
\end{equation}
This condition is satisfied when 
$
2 \beta \leq 1-2\alpha$ and $2 \gamma+1 \leq 2 \kappa.$
At last, Proposition \eqref{prop:prior_majo}  again ensures that \eqref{eq:bound_sieve} is true as soon as 
$$
k_n^2 \log^\rho k_n \wedge k_n n^{\mu_s} \gtrsim n \epsilon_n^2
$$
and we deduce from \eqref{eq:cond_kn} that
\begin{equation*}
 2 \beta = 1-2\alpha \qquad \text{and}\qquad  -\rho/2 + \kappa \leq \gamma \leq -1/2+\kappa. 
\end{equation*}
Moreover, we also see that $\beta+\mu_s \geq 1-2\alpha$, hence $\mu_s \geq 1/2-\alpha$, and if $\mu_s = 1/2-\alpha$ then $\gamma+\zeta \geq 2\kappa$; the former condition on $\mu_s$ yields $\mu_s \geq 1/2 - \alpha \geq \frac{1}{2s+2}$ (which naturally drives us to set $\mu_s = 1/4$ (case $s=1$) for adaptive prior).

We split the proof according to the adaptive or non adaptive case. 
\paragraph{Adaptive prior}
We first set $\mu$ independent of $s$ and equal to $1/4$. For any $s \in [1,3]$, we see that $\alpha(s)=s/(2s+2)$ is the admissible largest value of $\alpha$ and $\alpha(s)=3/8<s/(2s+2)$ as soon as $s > 3$. The corresponding value of $\beta$ is $1/(2s+2)$ when $s \in [1,3]$ and $\beta=1/8$ otherwise. Any choice of $\zeta \in [3/2, 2)$ permits to deal with the conditions on $\zeta$ that appears when $s=1$ or $s\geq 3$. 
The other values of $\gamma$ and $\kappa$ may be determined with respect to $\rho$. For instance, if we choose $\rho\in (1,2)$, we can take $\kappa =1$ and $\gamma=1/2$.
\paragraph{Non adaptive prior}
The non adaptive case is much more simpler since it is sufficient to fix 
$$
\mu_s = 1-2\alpha = 2/(2s+2)
$$
and $\zeta = 0$ to obtain suitable calibrations for $\alpha,\beta,\kappa$ and $\gamma$. This achieves the proof.
\end{proof}

\section{Concluding remarks}

In this paper, we exhibit a suitable prior which enable to obtain a contraction rate of the posterior distribution near the true underlying distribution $\mathbb{P}_{f^0,g^0}$. Moreover, this rate is polynomial with the number $n$ of observations, even if our SIM is an inverse problem with unknown operator of translation which depends on $g$.
From a technical point of view,  the keystones of such results are the tight link between the white noise model and the Fourier expansion as well as the smoothness of Gaussian law which permits to obtain an efficient covering strategy.

A natural problem would study of the behaviour of the posterior distribution regarding the functional objects shape $f^0$ and mixture law $g^0$. This question is tackled in \cite{BG_2} where we establish a contraction of  the posterior distribution  around $f^0$ and $g^0$ up to identifiability conditions. 

Another interesting extension would consider the SIM with a noise level $\sigma$ depending on $n$ in the Bayesian framework. This asymptotic setting is linked to the work of \cite{BG12} in which their $J$ curves are sampled at the $n$ points of a discrete design in $[0, 1]$.

At last, an open and challenging question concerns the research of stochastic algorithm to approach the posterior distribution in our non parametric Shape Invariant Model. 
One may think of an adaptation of the SA-EM strategy proposed in \cite{AKT} even if this approach is at the moment valid only in a parametric setting.

\appendix

\section{Topology on probability space}\label{section:appendix_proba}

\paragraph{Probability distances}

We study consistency using standard distance over probability measures. If $P$ and $Q$ are two probability measures over a set $X$, absolutely continuous with respect to a reference measure $\lambda$, $d_H$ refers to the Hellinger distance defined as
$$
d_H(P,Q) := \sqrt{\int_X  \left[ \sqrt{\frac{dP}{d\lambda}} - \sqrt{\frac{dQ}{d\lambda}}  \right]^2 d\lambda}.
$$

Note that $d_H$ does not depend on the choice of the dominating measure $\lambda$, and that the definition can be extended to any finite measures $P$ and $Q$ in a straightforward way.

When needed, we use the Total Variation distance between two probability measures $P$ and $Q$. If $\cB$ is the
$\sigma$-algebra of measurable sets with the reference measure $\lambda$, this distance is given by
$$
d_{TV}(P,Q) : = \sup_{A \in \cB} | P(A) - Q(A)| = \frac{1}{2} \int_X \left| \frac{dP}{d\lambda}
- \frac{dQ}{d\lambda} \right| d\lambda.$$
At last, we recall the definition of the Kullback-Leibler divergence (entropy) between $P$ and $Q$ since it is sometimes be used in the work:
$$d_{KL}(P,Q): = \int_{X} - \log \frac{dQ}{dP} dP.$$ In the sequel, we shall also use $V(P,Q)$ defined as a second order moment associated to the Kullback-Leibler divergence
$$
V(P,Q) := \int_{X} \left(\log \frac{dQ}{dP}\right)^2 dP.
$$
It may be reminded  the classical Pinsker's inequality 
\begin{equation}
\sqrt{\frac{1}{2}d_{KL}(P,Q) } \geq d_{TV}(P,Q),
\end{equation}
as well as
\begin{equation}\label{eq:dvt_dh}
\tfrac{1}{2}\, d_H(P,Q)^2 \leq d_{TV}(P,Q) \leq d_H(P,Q).
\end{equation}

\paragraph{Model Complexity}
To obtain the posterior consistency and convergence rate, we shall use results given by Theorem 2.1 of \cite{GGvdW00} which is stated below. This theorem exploits the notion of complexity of the studied model, and this complexity is traduced according to packing or covering numbers. 
For any set of probability measures $\cP$ endowed with a metric $d$, $D(\epsilon,\cP,d)$  refers to the $\epsilon$-packing number (the maximum number of points in $\cP$ such that the minimal distance between each pair is larger than $\epsilon$). 
The $\epsilon$-covering number  $N(\epsilon,\cP,d)$ is the minimum number of balls of radius $\epsilon$ needed to cover $\cP$. These two numbers are linked through the following inequality
$$
N(\epsilon,\cP,d) \leq D(\epsilon,\cP,d) \leq N(\epsilon/2,\cP,d) .
$$
At last, for $d$ a metric on finite measures, an $\epsilon$-bracket is a set of the form
$$
[L, U] := \left\{P \text{ s.t. } \frac{dL}{d\lambda} \leq \frac{dP}{d\lambda} \leq \frac{dU}{d\lambda} \right\},
$$
for $L$ and $U$ two finite measures such that $d(L, U)\leq \epsilon$ and $\lambda$ any dominating measure. The $\epsilon$-bracketing number $N_{[]}(\epsilon,\cP,d)$ is the minimal number of $\epsilon$-brackets 
needed to cover $\cP$. Note that $N_{[]}(\epsilon,\cP,d_H)$ is an upper bound of the $(\epsilon/2)$-covering number $\N(\epsilon/2,\cP,d_H)$. The bracketing entropy is then defined by $H_{[]}(\epsilon,\cP,d) := \log N_{[]}(\epsilon,\cP,d)$.

\section{Tools for the proof of Theorem 2.2}

\subsection{Entropy estimates}

\begin{proof}[Proof of Proposition \ref{prop:atheta}]
The proof is similar to Lemma 1 of \cite{GW}, we set $p=2\ell+1$ and for any $\epsilon >0$, we are going to build an explicit bracketing of $\cA_{\theta}$ and then bound $N_{[]}(\epsilon,\cA_{\theta},d_H)$.
For 
an integer $K$ which will be chosen in the sequel, we define $[\varphi^i_-,\varphi^i_+]$ of size $\Delta_\varphi = 1/K$, with $\varphi^i_- = (i-1) \Delta_\varphi$ and $\varphi^i_+ = i \Delta_\varphi$. 
For any $\delta>0$, we consider the lower and upper brackets
\[ l_i := (1+\delta)^{-1} \gamma_{\theta\bullet\varphi^i_-,(1+\delta)^{ - \alpha} Id} \qquad \text{and} \qquad u_i := (1+\delta) \gamma_{\theta\bullet\varphi^i_-,(1+\delta)^{ \alpha} Id}. \]
We are looking for some admissible values of $\alpha$, $\delta$, and $K$ such that the set $([l_i,u_i])_{i = 1 \ldots K}$ is an $\epsilon$-bracket of $\cA_{\theta}$ for $d_H$.
Of course, for all $\varphi \in [\varphi^i_-,\varphi^i_+]$, 
$l_i \leq \gamma_{\theta\bullet\varphi,Id}(.) \leq u_i$ should hold, but we can check that $\forall x \in \C$, 
\[  \frac{l_i(x)}{\gamma_{\theta\bullet\varphi,Id}(x)} \leq \frac{1}{1+\delta} \frac{1}{(1+\delta)^{-p\alpha}} e^{\frac{\|\theta\bullet\varphi -\theta\bullet\varphi_i^-\|^2 }{1-(1+\delta)^{-\alpha}}} \leq (1+\delta)^{p\alpha-1} e^{\frac{4 \pi^2 \Delta_\varphi^2 \|\theta\|_{\cH_1}^2}{1-(1+\delta)^{-\alpha}}}. \]
Hence, we must have $\alpha \leq 1/p$,
and we must also satisfy
$$
|\Delta_\varphi|^2 \leq \frac{1-p\alpha}{4 \pi^2 \|\theta\|_{\cH_1}^2} \left(1-(1+\delta)^{-\alpha}\right) \log(1+\delta) = \frac{\alpha (1-p\alpha) \delta^2}{4 \pi^2 \|\theta\|_{\cH_1}^2} \left(1 +o(1)\right),
$$
 where $o(1)$ does not depend on $p$ and goes to zero as $\delta\rightarrow 0$ uniformly in $\alpha$ in any positive neighbourhood of zero. 
 In a same way considering $ \gamma_{\theta\bullet\varphi, Id} u_i^{-1}$, we obtain
$$
\forall x \in \C \qquad \frac{\gamma_{\theta\bullet\varphi,Id}(x)}{u_i(x)} \leq (1+\delta)^{\alpha p -1 } e^{ \frac{4 \pi^2 \Delta_\varphi^2 \|\theta\|_{\cH_1}^2}{(1+\delta)^{\alpha}-1}},
$$ and the same conditions arise. In order to minimize the cardinal of the bracketing, $\Delta_\varphi$ must be as large as possible, we then maximize $\alpha (1-p\alpha)$ and choose $\alpha = (2p)^{-1}$.

We must now check that $d_H(l_i,u_i) \leq \epsilon$. Rapid computations show that
$$
d_H(l_i,u_i)^2 = \delta^2 + d_H(\gamma_{\theta\bullet\varphi_-^i,(1+\delta)^{-\alpha}Id}(.),\gamma_{\theta\bullet\varphi_-^i,(1+\delta)^{\alpha}Id}(.))^2.
$$
Using standard formula on Hellinger distance for multivariate gaussian laws,  we obtain
\begin{align*}
d_H(l_i,u_i)^2 &= \delta^2 + 2\left[ 1-\frac{2^p}{\left((1+\delta)^{\alpha}+(1+\delta)^{-\alpha} \right)^p}\right]  \\
&= \delta^2 + 2 \left[ 1-\frac{2^p\sqrt{1+\delta}}{\left( 1+(1+\delta)^{1/p}\right)^p}\right].
\end{align*}
One can easily check that, whatever $p\geq 1$, $\left( 1+(1+\delta)^{1/p}\right)^p \leq 2^p e^{\delta/2}$, which yields
$$
d_H(l_i,u_i)^2 \leq \tfrac{3}{2} \delta^2 + o(\delta^2) \leq 2 \delta^2
$$
for $\delta$ small enough. An admissible  choice of  $\delta$  should be $\delta = \epsilon/\sqrt{2}$, which insures $d_H(l_i,u_i) \leq \epsilon$. We then obtain
$$
\Delta_\varphi^2 \leq \frac{\delta^2 + o(\delta^2)}{16 \pi^2 p \|\theta\|_{\cH_1}^2} = \frac{ \epsilon^2+o(\epsilon^2)}{32 \pi^2 p \|\theta\|_{\cH_1}^2},
$$
where $o(\epsilon^2)$ does not depend on $p$. The number of brackets is now $K=\Delta_\varphi^{-1}$, this ends the proof of the proposition.
\end{proof}

\begin{proof}[Proof of Proposition \ref{prop:Pf}]
 We first fix the notation $p=2 \ell+1$ which refers to the dimension of the multivariate mixture. For any $R>0$ which will be chosen later, $\cE_{R}$ is the ball of in $\C^p$ of radius $R$. For sake of simplicity, we will sometimes omit the dependence on $\epsilon$ with the notation $p$.
 According to the hypotheses in Proposition \ref{prop:Pf}, there exists an absolute constant $a$ such that $\|\theta\| \leq w \leq a \sqrt{p}$. 
We first write
$$
d_{TV}(\PP_{\theta,g},\PP_{\theta,\tilde{g}}) \leq \frac{1}{2} \underbrace{\int_{\cE_R^c } \left| d\PP_{\theta,g} - d\PP_{\theta,\tilde{g}}  \right| (z)}_{:=(A)} +\frac{1}{2} \underbrace{\int_{\cE_R } \left| d\PP_{\theta,g} - d\PP_{\theta,\tilde{g}}  \right| (z)}_{:=(B)}.
$$
Let $\nu$ be a measure on $[0,1]$ that dominates both $g$ and $\tilde{g}$.

\paragraph{Term $(A)$}
We will pick $R$ such that $(A)$ is smaller than $\epsilon/2$, first set $R^2>(1+a)^2 p\geq a^{-2} (1+a)^2 \|\theta\|^2$ and with this choice,
$$
\forall \varphi \in [0,1]  \quad \forall z \in \cE_R^c \qquad \|z-\theta \bullet \varphi\| > \|z\|/(1+a).
$$
This simply implies that, 
\begin{eqnarray*}
(A)
&\leq& \pi^{-p} \int_{\cE_R ^c } \int_{0}^1 e^{-\frac{\|z\|^2}{(1+a)^2}}  \left| \frac{d g}{d \nu}(\varphi) - \frac{d \tilde{g}}{d \nu}(\varphi)\right| d \nu(\varphi) dz \\ & \leq& 2 (1+a)^{2p}\, \PR\left(\chi_{2p}^2 \geq \frac{2 R^2}{(1+a)^2}\right). 
\end{eqnarray*}
To deal with we last term we use a concentration of chi-square statistics inequality (see Lemma 1 of \cite{IL06}): for any $k\geq 1$ and $c>0$,
\begin{equation} \label{eq:chi_square}
 \PR\left( \chi_k^2 \geq (1+c) k \right) \leq \frac{1}{c\sqrt{2\pi}} e^{ -\frac{k}{2} [c -\log(1+c)] - \frac{1}{2} \log k}.
\end{equation}
Therefore, writing $R^2 = (1+a)^2 (1+c) p$ for $c>0$, one gets
\[ (A) \leq  \frac{1}{c\sqrt{\pi}} e^{-p [c - \log (1+c) - 2\log (1+a)] - \frac{1}{2}\log p }  \]
and this term is smaller than $\epsilon/2$ if we pick $c$ large enough, since $\log \frac{1}{\epsilon} \lesssim p$. 

\paragraph{Term $(B)$} \label{term_B}
We then consider $(B)$, following the strategy of \cite{GvdW01} which exploits the smoothness of Gaussian densities. We will exhibit a discrete mixture law which will be close to $\PP_{\theta,g}$, for any given $g$. Taylor's  expansion theorem yields:
\begin{equation}\label{eq:expo}
\forall k \in \N \quad 
\forall y \in \R_+ \qquad\underbraceabs{e^{-y} - \sum_{j=0}^{k-1} \frac{(-y)^j}{j!}}{:=R_k(y)}
\leq \frac{|y|^k}{k!} \leq \frac{(e |y|)^k}{k^k}.
\end{equation}
Thus, for all $z \in \cE_R$, we have
\begin{align*}
\PP_{\theta,g}(z) - \PP_{\theta,\tilde{g}}(z) 
& = \pi^{-p}\int_{0}^1  e^{-\|z-\theta \bullet \varphi\|^2} \left[ \frac{d g}{d \nu}(\varphi) - \frac{d \tilde{g}}{d \nu}(\varphi)\right] d \nu(\varphi) \\
& =  \pi^{-p}\sum_{j=0}^{k-1} \frac{(-1)^j}{j!}\int_{0}^1  \|z-\theta \bullet \varphi\|^{2j} \left[ \frac{d g}{d \nu}(\varphi) - \frac{d \tilde{g}}{d \nu}(\varphi)\right] d \nu(\varphi)  \\ & 
+ \pi^{-p}\int_{0}^1 R_k \left(\|z-\theta \bullet \varphi\|^2 \right)   \left[ \frac{d g}{d \nu}(\varphi) - \frac{d \tilde{g}}{d \nu}(\varphi)\right] d \nu(\varphi). 
\end{align*}
We now decompose  $\theta = (\theta_{-\ell}, \ldots, \theta_{\ell})$ and $z=(z_{-\ell}, \ldots, z_{\ell})$ using polar coordinates: $\theta_m = \rho^{(1)}_m e^{\i \alpha_m}$ and $z_m = \rho^{(2)}_m e^{\i \beta_m}$ for $|m| \leq \ell$. This leads to 
\[ \|z-\theta \bullet \varphi\|^{2} 
 = \|z\|^2 + \|\theta\|^2- 2 \sum_{m=-\ell}^{\ell} \rho^{(1)}_m \rho^{(2)}_m \cos (\beta_m - \alpha_m - m \varphi).
\]
For any integer $j \leq k$, we deduce that
$$
\|z-\theta \bullet \varphi\|^{2j} = C_j(z,\theta) + \sum_{r=1}^j \sum_{m=-\ell}^{\ell} a_{r,m}(z,\theta) \left[\cos (\beta_m - \alpha_m - m \varphi)\right]^r,
$$
where $(a(r,m))_{r = 1 \ldots k, m=-\ell \ldots \ell}$ is a complex matrix which only depends on $z$ and $\theta$. Using Euler's identity,
\begin{align*}
\|z-\theta \bullet \varphi\|^{2j} 
& = C_j(z,\theta) + \sum_{r=-j \ell }^{j \ell }b_{r}(z,\theta) e^{\i r \varphi},
\end{align*}
where $b$ stands for a complex vector obtained by the Binomial formula and coefficients $a_{r,m}(z,\theta)$. Consequently, for all $z \in \cE_R$
\begin{align*}
  \left(\PP_{\theta,g} - \PP_{\theta,\tilde{g}}\right)(z)
  &= \pi^{-p}\sum_{j=0}^{k-1}\frac{(-1)^j}{ j!} \int_{0}^{1} \bigg[ C_j(z,\theta) \\ &\quad + \sum_{r=-j \ell }^{j \ell }b_{r}(z,\theta) e^{\i r \varphi}\bigg] \left[ \frac{d g}{d \nu}(\varphi) - \frac{d \tilde{g}}{d \nu}(\varphi)\right] d \nu(\varphi) \\
  &\quad +\pi^{-p} \int_{0}^1 R_k \left(\|z-\theta \bullet \varphi\|^2\right)   \left[ \frac{d g}{d \nu}(\varphi) - \frac{d \tilde{g}}{d \nu}(\varphi)\right] d \nu(\varphi) \\
  & = \pi^{-p}\sum_{j=0}^{k-1}\frac{(-1)^j}{ j!} \bigg[ C_j(z,\theta) c_0(g-\tilde{g}) +\! \sum^{j \ell }_{r=-j \ell} b_{r}(z,\theta)c_{r}(g-\tilde{g})\bigg]\\ 
  &\quad + \pi^{-p}\int_{0}^1 R_k \left(\|z-\theta \bullet \varphi\|^2 \right) \left[ \frac{d g}{d \nu}(\varphi) - \frac{d \tilde{g}}{d \nu}(\varphi)\right] d \nu(\varphi). 
\end{align*}
Caratheodory's theorem shows that one can find  $\tilde{g}$ with a finite support of size $2 (k-1) \ell+1 \sim 2 k \ell$ such that
$$
\forall r \in [-(k-1) \ell, (k-1) \ell] \qquad c_r(g) = c_r(\tilde{g}).
$$
For such finite mixture law $\tilde{g}$, we obtain $\forall z \in \C^p$, 
\[ \PP_{\theta,g}(z) - \PP_{\theta,\tilde{g}}(z)  = \pi^{-p}\int_{0}^1 R_k \left(\|z-\theta \bullet \varphi\|^2\right) \left[ \frac{d g}{d \nu}(\varphi) - \frac{d \tilde{g}}{d \nu}(\varphi)\right] d \nu(\varphi), \]
and of course
\begin{align*}
(B) &\leq \pi^{-p} \int_{\cE_R} \left| \int_{0}^1 R_k \left(\|z-\theta \bullet \varphi\|^2\right) \left[ \frac{d g}{d \nu}(\varphi) - \frac{d \tilde{g}}{d \nu}(\varphi)\right] d \nu(\varphi) \right| dz \\
& \leq 2 \pi^{-p} \sup_{z \in \cE_R,\varphi \in (0,1)} R_k\left(\|z-\theta\bullet\varphi\|^2\right) Vol(\cE_R).
 \end{align*}
According to the choice $R = (1+a) \sqrt{(1+c) p}$ which implies that $\|z-\theta \bullet \varphi\| \leq (1+2a) \sqrt{(1+c) p}$, and using the volume of $\cE_R$ and Stirling's formula, we obtain
\begin{align*}
(B) &\lesssim \pi^{-p} \frac{\left(e (1+2a)^2 (1+c) p\right)^{k}}{k^k} \frac{\pi^p[ (1+a)^2 (1+c) p]^{p} }{p!} \\
& \lesssim  C_1 ^ p C_2 ^k e^{ - k \log(k) + k \log(p)},\\
\end{align*}
where we used in the last equation $p^p/p! \leq C^p$.
If we define the threshold $k$ in \eqref{eq:expo} such that $k \sim b \ell$ for a sufficiently large $b$, we then obtain for a universal $C$:
$$
(B) = 
\int_{\cE_R } \left| d\PP_{\theta,g} - d\PP_{\theta,\tilde{g}}  \right| (z) \lesssim e^{\ell (C - b   \log(b))}.
$$
In order to bound $(B)$ by $\epsilon/2$, we thus choose $k_{\epsilon} \sim b \ell_{\epsilon}$ for a sufficiently large absolute constant $b$. For such a choice, since $\log \frac{1}{\epsilon} \lesssim \ell_{\epsilon}$ we have found $\tilde{g}$ with a discrete support of cardinal $s_{\epsilon} \sim 2 b \ell_{\epsilon}^2$ points, with $s_\epsilon$ not depending on $g$, such that
$$
d_{TV}(\PP_{f,g},\PP_{f,\tilde{g}}) \leq \epsilon/2.
$$
Now, the first inequality in Proposition \ref{prop:Pf} comes from Proposition \ref{prop:Mktheta}.

The second inequality in Proposition \ref{prop:Pf} is proved from the first one, using the relation $\|\theta\|_{\cH_1} \leq \ell \|\theta\|$ valid for any $f \in \cH_s^\ell$.
\end{proof}

\begin{proof}[Proof of Lemma \ref{lemma:dVT_sur_f}]
 We follow a straightforward argument: $\PP_{f, g}$ is a mixture model so 
 \[ \PP_{f, g} = \int_{0}^1\PP_{f, \delta_\alpha} d g(\alpha). \]
Thus
 \begin{align*}
  d_{TV}\left( \PP_{f, g}, \PP_{\tilde{f}, g} \right) &= \left\| \int_{0}^1 \left( \PP_{f, \delta_\alpha} - \PP_{\tilde{f}, \delta_\alpha} \right) d g(\alpha)\right\|_{TV} \\
  &\leq  \int_{0}^1  \left\| \PP_{f, \delta_\alpha} - \PP_{\tilde{f}, \delta_\alpha} \right\|_{TV} d g(\alpha) \\
  &= \left\| \PP_{f, \delta_0} - \PP_{\tilde{f}, \delta_0} \right\|_{TV} \leq d_H\left(\PP_{f, \delta_0}, \PP_{\tilde{f}, \delta_0}\right).
 \end{align*}
 Assume now $Y\sim \PP_{f, \delta_0}$, hence from \eqref{eq:model} $d Y = f(x) d x + d W$, with $W$ is a complex standard Brownian motion. If we denote $U$ a random variable $\cN_\C(0,1)$, standard argument using  Girsanov's formula yields
 \begin{align*}
  d_H^2\left(\PP_{f, \delta_0}, \PP_{\tilde{f}, \delta_0}\right) &= 2 \left( 1- \EE_{f, \delta_0} \sqrt{ \frac{ d\PP_{\tilde{f}, \delta_0} }{ d\PP_{f, \delta_0} }(Y)} \right) \\
  &= 2 \left( 1- \EE_{f, \delta_0} \sqrt{ \exp\left( 2\re\langle \tilde{f}-f, d W\rangle - \|\tilde{f}-f\|^2 \right)} \right) \\
  &= 2 \left( 1- \exp\left(\frac{-\|\tilde{f}-f\|^2}{2}\right) \EE_{U}\left[\exp\left(\|\tilde{f}-f\| \re(U)\right)\right] \right)  \\
  &= 2 \left( 1- \exp\left(\frac{-\|\tilde{f}-f\|^2}{4}\right) \right) \leq \frac{\|\tilde{f}-f\|^2}{2}.
 \end{align*}
\end{proof}

\subsection{Link between Kullback-Leibler and Hellinger neighbourhoods}

\begin{proof}[Proof of Proposition \ref{prop:appli_wong_shen}]
This proposition uses a corollary of Rice's formula (see \cite{AW09} for various applications of such formula), stated in Lemma \ref{lemma:rice} and postponed after this proof.

We begin with Girsanov's formula \eqref{eq:Girsanov}. 
Write now $Y=f^{0, -\tau}+W$ where $W$ stands for a complex standard Brownian motion independent of the random shift $\tau$ (whose law is $g^0$). The $L^2$ norm is invariant with any shift thus
\begin{align*}
\frac{d\PP_{f^0,g^0}}{d\PP_{f,g}}(Y) &= \exp \left(\|f\|^2 - \|f^0\|^2\right)
\frac{\int_{0}^1 e^{2 \re\langle f^{0, -\alpha_1}, f^{0, -\tau}+ d W \rangle} d g^0(\alpha_1) }{\int_{0}^1 e^{2\re\langle f^{-\alpha_2}, f^{0, -\tau}+ d W \rangle} d g(\alpha_2)} \\ 
& \leq \exp \left(\|f\|^2 - \|f^0\|^2\right) \exp \left(2 \sup_{\alpha_1,\alpha_2}\re \langle f^{0, -\alpha_1} - f^{-\alpha_2},f^{0, - \tau} \rangle \right) \\
& \exp \left(2 \sup_{\alpha_1,\alpha_2} \re\langle f^{0, -\alpha_1} - f^{-\alpha_2} ,d  W \rangle \right) \\
& \leq e^{(\|f\| + \|f^0\|)^2} e^{Z_1+Z_2},
\end{align*}
where the last inequality is obtained using Cauchy-Schwarz's inequality and the notations
\begin{align*}
 Z_1&:=2\sup_{\alpha_1} \re\langle f^{0, -\alpha_1} , d  W \rangle  = 2 \sup_{\alpha_1} \re\int_{0}^1 \overline{f^0}(s-\alpha_1) d W_s, \\
 Z_2&:=2\sup_{\alpha_2} \re\langle -f^{-\alpha_2} , d  W \rangle = 2 \sup_{\alpha_2} \re\int_{0}^1 -\overline{f}(s-\alpha_2) d W_s.
\end{align*}
We now set $\delta \in (0,1]$  (it will be precisely fixed in the sequel)  and we define the trajectories $\cE_{\delta}$ as
$$
\cE_\delta := \left\{ Y = f^{0, - \tau}+ W \quad | \quad\frac{d\PP_{f^0,g^0}}{d\PP_{f,g}}(Y) \geq e^{1/\delta}\right\}.
$$
Hence, following the definition of $M_{\delta}^2$ of \eqref{eq:condition_moment}, we have
$$
M_{\delta}^2= \EE_{Y \sim \PP_{f^0,g^0}} \left[ \left( \frac{d\PP_{f^0,g^0}}{d\PP_{f,g}}(Y)\right)^{\delta} \1_{Y \in \cE_\delta} \right].
$$
For $\delta$ small enough, ($\delta \leq \frac{1}{2\left(\|f\| + \|f^0\| \right)^2})$:
\begin{eqnarray*}
 M_{\delta}^2 &\leq& e^{\delta (\|f\| + \|f^0\| )^2}  \EE e^{\delta (Z_1+Z_2)} \1_{Z_1+Z_2 \geq \frac{1}{\delta}- (\|f\| + \|f^0\| )^2} \\ &
 \leq & e^{\delta (\|f\| + \|f^0\|)^2} \EE e^{\delta (Z_1+Z_2)} \1_{Z_1+Z_2 \geq \frac{1}{2\delta}} \\ &
 \leq &e^{\delta (\|f\| + \|f^0\|)^2} \EE e^{\delta (Z_1+Z_2)} \1_{e^{\delta (Z_1+Z_2)}  \geq \sqrt{e}}.
\end{eqnarray*}
Integrating by parts the last expectation, the use of Lemma \ref{lemma:rice} yields
\begin{align}\nonumber
M_{\delta}^2 &\leq e^{\delta (\|f\| + \|f^0\|)^2} \int_{\sqrt{e}}^{+ \infty} \PR \left(e^{\delta (Z_1+Z_2)} > u \right) d u \\\nonumber
& =  e^{\delta (\|f\| + \|f^0\|)^2} \int_{\sqrt{e}}^{+ \infty} \left[\PR \left(\frac{Z_1}{2} \geq \frac{\log u}{4\delta} \right) + \PR \left(\frac{Z_2}{2} \geq \frac{\log u}{4\delta} \right)\right] d u \\\label{eq:rice_application}
M_{\delta}^2 & \leq  C(f^0,f)  e^{\delta (\|f\| + \|f^0\|)^2} \int_{\sqrt{e}}^{+ \infty} \left[ e^{- \frac{\log^2(u)}{16\delta^2\|f^0\|^2}}  + e^{- \frac{\log^2(u)}{16\delta^2\|f\|^2}} \right] d u.
\end{align}
Now, we can choose $\delta$ non negative and small enough such that  $M_\delta^2 < \infty$ since for $u \geq \sqrt{e}$, we have

$$
e^{-\frac{\log^2(u)}{16\delta^2 \|f^0\|^2}} \leq e^{- \frac{\log(u)}{32 \delta^2 \|f^0\|^2}} = u^{-1/{32 \delta^2 \|f^0\|^2}},
$$
which is an integrable function as soon as
$\delta^2 < \frac{1}{32 \|f^0\|^2}$, and the same holds with $f$ instead of $f_0$. Note that $M_{\delta}^2$ is uniformly bounded if $f$ is picked into a ball centered at $0$ with radius $2\|f^0\|$.
\end{proof}

We now show that the technical inequality used in \eqref{eq:rice_application} is satisfied.

\begin{lemma}\label{lemma:rice}

 Let $W$ a complex standard Brownian motion and $u$ a complex $1$-periodic map of $\cH_s$. We assume that $u$ is of class $\cC^2$. Then when $t/\|u\| \longrightarrow +\infty$, we have
 $$ \PR\left( \sup_\alpha \re\langle u^{-\alpha},  d W\rangle >t\right) \lesssim \frac{\|u'\|}{2\pi \|u\|} \exp\left(\frac{-t^2}{\|u\|^2}\right).$$
  In particular, if $u \in \cH_s^\ell$, we have
  $$ \PR\left( \sup_\alpha \re\langle u^{-\alpha}, d  W\rangle > t\right) \lesssim \frac{\ell}{2\pi} \exp\left(\frac{-t^2}{\|u\|^2}\right).$$
 \end{lemma}

\begin{proof}
We define the following process
$$
\forall \alpha \in [0,1] \qquad 
X(\alpha) := \frac{\sqrt{2}\re\left(\int_{0}^1 \overline{u}(s-\alpha) d W_s\right)}{\|u\|}.
$$
$X$ is a Gaussian centered process. Its covariance function is given by
$$
\Gamma(t) = \EE \left[X(0) X(t)\right].
$$
Obviously, one has $\Gamma(0)=1$ and Cauchy-Schwarz's inequality implies that  $\Gamma(s) \leq \Gamma(0)$. Moreover, since $\Gamma$ is $\mathcal{C}^1([0,1])$, we deduce that
$\Gamma'(0) = 0$ and simple computation yields
$$
\Gamma"(0)= \frac{\re\left(\int_{0}^1 u'(s) u"(s) ds\right)}{\|u\|^2} = - \frac{\|u'\|^2}{\|u\|^2}.
$$
Rice's formula (see for instance exercice 4.2, chapter 4 of \cite{AW09}) then yields that when 
$t \longrightarrow +\infty$, we have
$$
\PR \left( \sup_{\alpha} X(\alpha) > t \right) \sim \frac{\|u'\|}{2 \pi \|u\|} e^{- t^2/2}.
$$
This ends the proof of the first inequality. Assume furthermore that $u \in \cH_s^\ell$,  Parseval's equality implies that $\|u'\| \leq \ell \|u\|$ and we obtain the second inequality. 
\end{proof}

\subsection{Hellinger neighbourhoods}

\begin{proof}[Proof of Proposition \ref{prop:E1}]
Recall first that if $Y$ follows $\PP_{f^0,g^0}$, one shift $\beta$ is randomly sampled according to $g^0$. Conditionally to this shift $\beta$, $Y$ is described trough a white noise model $d Y(x) = f^0(x-\beta) d x + d W(x)$. For any function $F$ of the trajectory $Y$, we will denote $\EE_\beta F(Y)$ the expectation of  $F(Y)$ up to the condition that the shift is equal to $\beta$, and of course one has 
 $$\EE_0[F(Y)] = \int_0^1 \EE_\beta[F(Y)] d g^0(\beta).$$
For each possible value of $\beta \in [0,1]$, we define 
\begin{align*}
D_\beta(\alpha) &:= \exp\left( 2\re\langle f_{ \ell_n}^{0,-\alpha}, f^{0,-\beta}\rangle + 2\re\langle f_{\ell_n}^{0,-\alpha}, d W\rangle - \|f_{\ell_n}^0\|^2\right), \\
 X_\beta(\alpha) &:= \exp\Big(2\re \langle (f^0 - f^0_{ \ell_n})^{-\alpha}, f^{0,-\beta}\rangle \\ &\quad\, + 2\re \langle (f^0-f^0_{\ell_n})^{-\alpha}, d W\rangle - \|f^0-f^{0}_ {\ell_n}\|^2\Big).
\end{align*}
 We can now split the randomness of the Brownian motion into two parts: the first one is spanned by the Fourier frequencies from $-\ell_n$ to $\ell_n$ and the second part is its orthogonal (in $L^2$): $W=W_1+W_2$. Of course,  $W_1$ and $W_2$ are independent.
 
 Moreover, $\langle f_{ \ell_n}^{0,-\alpha}, d  W\rangle = \langle f_{\ell_n}^{0,-\alpha}, d  W_1\rangle$ and $\langle (f^0-f^0_{ \ell_n})^{-\alpha}, d W\rangle = \langle (f^0-f^0_{ \ell_n})^{-\alpha}, d W_2\rangle$. For any fixed $\beta$, $D_\beta(\alpha)$ is measurable with respect to the filtration associated to  $W_1$, and $X_\beta(\alpha)$ is independent of $W_1$.
We thus obtain using Jensen's inequality and this filtration property that
 \begin{align*}
(\tilde{E_1})^2 &= \displaystyle \EE\left[ \log \frac{\int_{0}^1 D_\beta(\alpha) X_\beta(\alpha) d g^0(\alpha)}{\int_{0}^1 D_\beta(\alpha) d g^0(\alpha)} \right] \\
  &\leq \displaystyle \log \int_0^1 \EE^{W_2}_\beta\left[ \EE^{W_1}_\beta\left[ \left. \frac{\int_{0}^1 D_\beta(\alpha) X_\beta(\alpha) d g^0(\alpha)}{\int_0^1 D_\beta(\alpha) d g^0(\alpha)} \right| W_2 \right]\right] d g^0(\beta) \\
  & \leq \displaystyle \log \int_0^1 \EE^{W_2}_\beta\left[ X_\beta(\alpha)  \EE^{W_1}_\beta\left[ \left. \frac{\int_0^1 D_\beta(\alpha) d g^0(\alpha)}{\int_0^1 D_\beta(\alpha) d g^0(\alpha)} \right| W_2 \right]\right] d g^0(\beta) \\
  &\leq \displaystyle\log \int_0^1 \left( \sup_\alpha \EE^{W_2}_\beta \left[ X_\beta(\alpha) \right]\right) d g^0(\beta) .
 \end{align*}
The notation $\EE_{\beta}^{W_1} F(Y)$ (resp. $\EE_{\beta}^{W_2} F(Y)$) used above refers to the expectation of $F(Y)$ with respect to $W_1$ (resp. with respect to $W_2$) with a fixed $\beta$.

Now, one should remark that $X_\beta(\alpha)$ has the same law as $$\exp\left(2\re\langle (f^0 - f^{0}_{\ell_n})^{-\alpha}, f^{0,-\beta}\rangle + U\right),$$ 
where $U\sim\cN_\R\left(-\|f^0-f^{0}_{\ell_n}\|^2, 2\|f^0-f^{0}_{\ell_n}\|^2\right)$, and $\EE\left[e^U\right]=1$. Hence
\begin{eqnarray*}
(\tilde{E_1})^2 & \leq &\log \int_{0}^1 \sup_{\alpha} \exp\left(2\re\langle (f^0 - f^{0}_{\ell_n})^{-\alpha}, f^{0,-\beta}\rangle \right) d g^{0}(\beta)\\ & \leq& \log \sup_{\alpha,\beta} \exp\left(2\re\langle (f^0 - f^{0}_{\ell})^{-\alpha}, f^{0,-\beta}\rangle \right) 
\end{eqnarray*}
 We can now switch $\log$ and $\sup$ since $\log$ is increasing, and we obtain
$$
 (\tilde{E_1})  
 \leq  \sqrt{2 \sup_{\alpha,\beta} \re\langle (f^0 - f^{0}_{\ell_n})^{-\alpha}, f^{0,-\beta}\rangle}. 
 $$
Again, we can use the orthogonal decomposition $f^{0,-\beta} = f^{0,-\beta}_{\ell_n} + f^{0,-\beta}- f^{0,-\beta}_{\ell_n}$ and Cauchy-Schwarz's inequality yields
$ (\tilde{E_1})  \leq \sqrt{2} \|f^{0}- f^{0}_{\ell_n}\|.$

Note that untill now we did not use the hypothesis $f^0\in \cH_s$. It is only needed to get the last inequality in Proposition \ref{prop:E1}.
\end{proof}

To establish Lemma \ref{lemma:mixture_lemma}, we first remind the following useful result.
\begin{lemma}\label{lemma:l1_gauss}
For any any dimension $p$ and any couple of points $(z_1,z_2) \in \C^p$, if $\|z_1-z_2\|$ is the Euclidean distance in $\C^p$, then one has
$$
d_{TV}(\gamma_{z_1}, \gamma_{z_2}) = \frac{1}{2} \|\gamma_{z_1} - \gamma_{z_2}\|_{L_1} = \left[ 2\Phi\left(\frac{\|z_1-z_2\|}{2} \right)-1\right]  \leq \frac{ \|z_1-z_2\|}{ \sqrt{2\pi}},
$$
where $\Phi$ stands for the cumulative distribution function of a real standard Gaussian variable.
\end{lemma}

\begin{proof}[Proof of Lemma \ref{lemma:mixture_lemma}]
Adapting the proof of Lemma 5.1 of \cite{GvdW01}, 
\begin{align*}
\|\PP_{f^0_{\ell_n},\check{g}} -\PP_{f^0_{\ell_n},\tilde{g}}\|_{L_1} &\leq \sum_{j=1}^J \int_{\varphi_j-\eta/2}^{\varphi_j+\eta/2}  \|\gamma_0(.-\theta\bullet\varphi) - \gamma_0(.-\theta\bullet\varphi_j)\|_{L_1} d \check{g}(\varphi) \\
& + 2 \sum_{j=1}^J \left|\check{g}([\varphi_j-\eta/2,\varphi_j+\eta/2]) - p_j\right|.
\end{align*}
Using Lemma \ref{lemma:l1_gauss} ends the proof.
\end{proof}

\begin{proof}[Proof of Proposition \ref{prop:finite_mixture}] \sloppy
The construction used in the proof of Proposition \ref{prop:Pf} provide a mixture $\tilde{\tilde{g}}$ such that $\tilde{\tilde{g}}$ is supported by $\tilde{J}_n := C \ell_n^2$ points (denoted $(\varphi_j)_{j=1\ldots \tilde{J}_n}$) so that $d_H(\PP_{f^0_{\ell_n},g^0} ,\PP_{f^0_{\ell_n},\tilde{\tilde{g}}}) \leq \epsilon_n$. 
Therefore $\tilde{\tilde{g}}= \sum_{j=1}^{\tilde{J}_n} w_j \delta_{\varphi_j}.$ 
As pointed by \cite{GvdW01}, one can slightly modify $\tilde{\tilde{g}}$ so that the support points are separated enough as follows. First, denote
 $(\psi_j)_{j=1 \ldots J_n}$ the subset of $(\varphi_j)_{j=1 \ldots \tilde{J}_n}$ 
which is $\eta_n$-separated with a maximal number of elements. Hence, $J_n \leq \tilde{J}_n$ and up to a permutation, 
one can divide $(\varphi_j)_{j=1 \ldots \tilde{J}_n}$ in two parts:
$ (\varphi_j)_{j=1 \ldots \tilde{J}_n} = (\psi_j)_{j=1 \ldots J_n} \cup  (\varphi_j)_{j=J_n+1 \ldots \tilde{J}_n}$. 
For any $i \in \{J_n+1, \ldots, \tilde{J}_n\}$, we define $\psi_{j(i)}$ as the closest point of $(\psi_j)_{j=1 \ldots J_n}$, the new discrete mixture law is then given by
$$
\tilde{g} = \sum_{j=1}^{J_n}\underbrace{\left(w_j +  \sum_{i > J_n \vert j(i) = j} w_{i} \right)}_{:=\tilde{w}_j} \delta_{\psi_j}.
$$
\fussy
Of course, $\tilde{g}$ as a support which is $\eta_n$-separated. Moreover, we have
\begin{multline*}
 2d_{TV}\left( \PP_{f^0_{\ell_n},\tilde{g}},\PP_{f^0_{\ell_n},\tilde{\tilde{g}}} \right) \\
 \begin{aligned}
  &= \int_{\C^{\ell_n}} \left| \sum_{i=1}^{J_n} \tilde{w}_i \gamma(z-\theta \bullet \psi_i)  - \sum_{i=1}^{\tilde{J}_n} w_i \gamma(z-\theta \bullet \varphi_i) \right| d z \\
  & = \int_{\C^{\ell_n}} \left| \sum_{j=1}^{J_n} (\tilde{w}_j -w_j) \gamma(z-\theta \bullet \psi_j) - \sum_{i > J_n} w_i \gamma(z-\theta \bullet \varphi_i) \right| d z\\
  & = \int_{\C^{\ell_n}} \left| \sum_{j=1}^{J_n} \sum_{i > J_n \vert j(i)=j}  w_i [\gamma(z-\theta \bullet \psi_j) - \gamma(z-\theta \bullet \varphi_i) ] \right| d z.
 \end{aligned}
\end{multline*}
Then, Fubini's theorem yields
$$
d_{TV}\left( \PP_{f^0_{\ell_n},\tilde{g}},\PP_{f^0_{\ell_n},\tilde{\tilde{g}}} \right)  \leq 
\sum_{j=1}^{J_n}
\sum_{i > J_n \vert j(i)=j}  w_i d_{TV} ( \gamma_{\theta \bullet \varphi_i},\gamma_{\theta \bullet \psi_j}),
$$
and we deduce from Lemma \ref{lemma:l1_gauss}  that
$$
d_{TV}\left( \PP_{f^0_{\ell_n},\tilde{g}},\PP_{f^0_{\ell_n},\tilde{\tilde{g}}} \right)  \leq \sqrt{2 \pi}
\sum_{j=1}^{J_n} \sum_{i > J_n \vert j(i)=j}  w_i \|\theta\|_{\cH_1}\eta_n \leq \sqrt{2 \pi} \|\theta\|_{\cH_1}\eta_n.
$$
Now the relations between Hellinger and Total Variation distances \eqref{eq:dvt_dh} yield
\[ d_H(\PP_{f^{0}_{\ell_n},g^0},\PP_{f^{0}_{\ell_n},\tilde{g}})\leq \epsilon_n+d_H(\PP_{f^{0}_{\ell_n},\tilde{g}},\PP_{f^{0}_{\ell_n},\tilde{\tilde{g}}}) \leq \left(1+(8\pi)^{1/4}\|\theta\|_{\cH_1}^{1/2}\right) \epsilon_n. \]
Lemma \ref{lemma:mixture_lemma} permits to conclude.
\end{proof}

\subsection{Checking the conditions of Theorem \ref{theo:posterior}}

\begin{proof}[Proof of Propostion \ref{prop:prior_mino}]

We have seen in the proof of Proposition \ref{prop:appli_wong_shen} that $M_\delta^2$ is uniformly bounded with respect to $\|f\|$ and $\|f^0\|$ for a suitable choice of $\delta$.  
We restrict our study to the elements $f$ such that $\|f\| \leq 2 \|f^0\|$.

We 
know from Proposition \ref{prop:appli_wong_shen} and Theorem \ref{theo:wong_shen} that as soon as $\tilde{\epsilon}_n \log \frac{1}{\tilde{\epsilon}_n} \leq c \epsilon_n$ with $c$ small enough:
\begin{eqnarray*}
\cV_{\tilde{\epsilon}_n}(\PP_{f^0,g^0},d_H)& := &\left\{ \PP_{f,g} \in \cP \vert d_H(\PP_{f^0,g^0} ,\PP_{f,g} ) \leq \tilde{\epsilon}_n  \,\text{and} \,  \|f\| \leq 2 \|f^0\|\right\}\\ & \subset &\cV_{\epsilon_n}(\PP_{f^0,g^0},d_{KL}).
\end{eqnarray*}

This last condition on $\tilde{\epsilon}_n$ is true as soon as
\begin{equation}\label{eq:def_tilde_epsilon}
\tilde{\epsilon}_n := \tilde{c} \epsilon_n \left( \log \frac{1}{\epsilon_n}\right)^{-1}
\end{equation}
with $\tilde{c}$ small enough. Now, Proposition \ref{prop:Hellinger_neighbourhood} permits to describe a subset of $\cV_{\tilde{\epsilon}_n}(\PP_{f^0,g^0},d_H)$, by the definition of subsets $\cF_{\tilde{\epsilon}_n}$ and $\cG_{\tilde{\epsilon}_n}$ for $f$ and $g$. Choose $\ell_n:=\tilde{\epsilon}_n^{-1/s}$. 

We first bound the prior mass on $\cG_{\tilde{\epsilon}_n}$. This follows from the lower bound given by Lemma \ref{lemma:mino_simplex}. The prior for $g$ is a Dirichlet process with a finite base measure $\alpha$ admitting a continuous positive density on $[0, 1]$. Since $\eta_n$ goes to zero, for $n$ large enough $\alpha(\psi_j - \eta_n/2, \psi_j+\eta_n/2)$ for any $j=1 \ldots J_n$. 
Note that $J_n \lesssim \ell_n^2 = \tilde{\epsilon}_n^{-2/s} \leq  \tilde{\epsilon}_n^{-2}$. Thus, there exists an absolute constant $a\in (0, 1]$ such that the condition $J_n \leq 2 (a\tilde{\epsilon}_n)^{-2}$ is fulfilled, and 
one can find universal constants $C$ and $c$ such that for $n$ large enough
\begin{equation}\label{eq:mino_G}
\Pi_n\left( \cG_{\tilde{\epsilon}_n} \right) \geq \Pi_n\left( \cG_{a\tilde{\epsilon}_n} \right) \geq C e^{-c J_{n} \log \frac{1}{\tilde{\epsilon}_n^2} } \geq C e^{-c \ell_n^2 \log \frac{1}{\tilde{\epsilon}_n}}.
\end{equation}

We next consider the prior mass on $\cF_{\tilde{\epsilon}_n}$. Remark that when $n$ is large enough, any element of  $\cF_{\tilde{\epsilon}_n}$ satisfies $\|f\| \leq 2 \|f^0\|$ and the additional condition on $\|f\|$ in the definition of $\cV_{\tilde{\epsilon}_n}(\PP_{f^0,g^0},d_H)$ is instantaneously fulfilled. Remark that from the construction of our prior on $f$, one has
$$
\Pi_n \left(\cF_{\tilde{\epsilon}_n} \right) \geq \lambda({\ell_n}) \times \pi_{\ell_n} \left(
B\left(\theta^0_{\ell_n}, \tilde{\epsilon}_n^2 \right) \right).
$$
From our assumption on the prior $\lambda$, we have
$\lambda(\ell_n) \geq e^{- c \ell_n^2 \log^\rho \ell_n}$, and the value of the volume of the $(4 \ell_n+2)$-dimensional Euclidean ball of radius $\tilde{\epsilon}_n^2$ implies
$$
\Pi_n \left(\cF_{\tilde{\epsilon}_n} \right) \geq  e^{- c \ell_n^2\log^\rho  \ell_n}
\inf_{u\in B
\left(0, \tilde{\epsilon}_n^2 \right)} \left(\frac{
e^{-\|\theta^0+u\|^2/\xikn^2} 
}{\pi^{2\ell_n+1} \xikn^{2(2\ell_n+1)} 
} \right)
\left( \tilde{\epsilon}_n^2 \right)^{4 \ell_n+2} 
\frac{\pi^{2 \ell_n+1}}{\Gamma(2 \ell_n+2)}.
$$
For $n$ large enough we get
\begin{align}\notag
\Pi_n \left(\cF_{\tilde{\epsilon}_n} \right) &\geq \exp -\left[ c \ell_n^2\log^\rho \ell_n +  \xikn^{-2} \right. \\ \notag &\quad \left.
 + (2\ell_n+1) \left(\log \ell_n + 4 \log(1/\tilde{\epsilon}_n) - \log \xikn^{-2}
 + \cO(1) \right)\right] \\ \label{eq:mino_F}
 &\geq \exp \left[ - (c+o(1))\, \left[ \ell_n^{2} \log^\rho \ell_n  \vee \xikn^{-2} \right]\right]
\end{align}
Gathering \eqref{eq:mino_G} and \eqref{eq:mino_F}, the relations $\ell_n=\tilde{\epsilon}_n^{-1/s}$ and \eqref{eq:def_tilde_epsilon} lead to
\begin{align*}
 \Pi_n \left( \cV_{\epsilon_n}(\PP_{f^0,g^0},d_{KL}) \right) &\geq \Pi_n \left(\cF_{\tilde{\epsilon}_n} \right) \Pi_n\left( \cG_{\tilde{\epsilon}_n} \right) \\
 &\geq \exp \left[ - (c +o(1))\, \left[ \ell_n^{2} \log^\rho \ell_n \vee \xikn^{-2} \right]\right] \\
 &\geq \exp \left[ - (c +o(1))\, \left[\tilde{\epsilon}_n^{-2/s} \log^\rho\left(1/\tilde{\epsilon}_n\right)  \vee \xikn^{-2} \right]\right] \\
 &\geq \exp \left[ - (c +o(1))\, \left[\epsilon_n^{-2/s} \left(\log (1/\epsilon_n)\right)^{\rho+2/s}  \vee \xikn^{-2} \right]\right]
\end{align*}
for constants $c>0$.
\end{proof}

\begin{proof}[Proof of Proposition \ref{prop:prior_majo}]
The upper bound on the packing number comes directly from Theorem \ref{theo:recouvrement} since we set $w_n = \sqrt{2k_n+1}$.

Now, to control the prior mass outside the sieve, remark first that owing to the construction of our prior, we have
\begin{equation}\label{eq:complement}
\Pi_n \left(\cP \setminus \cP_{k_n,w_n}\right)  \leq \sum_{|k| \geq k_n} \lambda(k) + \Pr \left( \sum_{|k| \leq k_n} |\theta_k|^2 \geq w_n^2 \right),
\end{equation}
where each $\theta_k$ for $-k_n \leq k \leq k_n$ follows a centered Gaussian law of variance $\xikn^2$.
Now, there exists some constants $c$ and $C$ such that for sufficiently large $n$:
$$
\sum_{|k| \geq k_n} \lambda(k) \leq C \lambda(k_n) \leq e^{- c k_n^2 \log^\rho(k_n)}.$$
Regarding now the second term of the upper bound in \eqref{eq:complement}, we use \eqref{eq:chi_square} to get
\begin{align*}
\Pr \left( \sum_{|k| \leq k_n} |\theta_k|^2 \geq w_n^2 \right) &=
\Pr \left( \sum_{|k| \leq k_n} \left|\frac{\theta_k}{\xikn}\right|^2  \xikn^2\geq w_n^2 \right) \\
& \leq  \PR\left( \chi_{2k_n+1}^2 \geq 2  (2k_n+1) \xikn^{-2} \right) \\
& \leq  \frac{1}{( \xikn^{-2}-1)\sqrt{\pi}} e^{-(2k_n+1)[ \xikn^{-2}-1-\log  \xikn^{-2}]-\log (2k_n+1)/2}.
\end{align*}
Now, using the value of $\xikn$, we obtain
$$
\Pi_n \left(\cP \setminus \cP_{k_n,w_n}\right)  \leq  e^{ 
- c [ k_n^2 \log^\rho(k_n) \wedge k_n \xikn^{-2}]}.
$$
This concludes the proof of the Proposition.
\end{proof}

 \section*{Acknowledgements} S. G. is indebted to Jean-Marc Aza\"is for stimulating discussions related to some technical parts of this work. Authors also thank J\'er\'emie Bigot, Isma\"el Castillo, Xavier Gendre, Judith Rousseau and Alain Trouv\'e for enlightening exchanges.

\bibliographystyle{alpha}
\bibliography{paper}

\end{document}